\documentclass[12pt, reqno]{amsart}

\usepackage{amsmath, amsfonts, amssymb, hyperref}
\usepackage{graphics}
\usepackage[usenames]{color}

{\theoremstyle{plain}
\newtheorem{theorem}{Theorem}[section]

\newtheorem{lemma}[theorem]{Lemma}
\newtheorem{corollary}[theorem]{Corollary}
}
{\theoremstyle{definition}
\newtheorem{definition}[theorem]{Definition}
\newtheorem*{definition*}{Definition}
\newtheorem{remark}[theorem]{Remark}
\newtheorem{example}[theorem]{Example}

\newtheorem{hypothesis}[theorem]{Hypothesis}
}

\numberwithin{equation}{section}
\newcommand{\dom}{\leq _{dom}}
\newcommand{\lex}{\leq _{lex}}
\newcommand{\sbar}{\overline \sigma}
\newcommand{\tbar}{\overline \tau}
\newcommand{\sslash}{\overline \sigma \backslash \overline \sigma _s}
\newcommand{\tslash}{\overline \tau \backslash \overline \tau _1}
\newcommand{\N}{\mathbb{N}}
 
\newcommand\ov{\overline} 
\newcommand\Kfunc[3]{K^{\sigma,\bar\sigma}_{\tau,\bar\tau}(#1,#2\,|\,#3)}
\newcommand\Hfunc[3]{H^{S}_{T}(#1,#2\,|\,#3)}   
\newcommand{\C}{\mathbb{C}}
\newcommand{\Z}{\mathbb{Z}}
\newcommand{\Gr}{Gr}
\newcommand\mucomp{{\mu^c}}
\newcommand\nucomp{{\nu^c}}
\newcommand\rhocomp{{\rho^c}}

\begin{document}
\title[Skew Schur function differences]{Schur positivity of skew Schur function differences and 
applications to ribbons and Schubert classes}

\author{Ronald C. King, Trevor A. Welsh and Stephanie J. van Willigenburg}
\email{rck@maths.soton.ac.uk}
\address{School of Mathematics,
         University of Southampton,
         University Road,
         Southampton,
         Hampshire,
         SO17 1BJ
         U.K.}

\email{taw@maths.soton.ac.uk}
\address{Department of Physics,
         University of Toronto,
         60 St.\ George Street,
         Toronto,
         Ontario,
         M5S 1A7
         Canada}

\email{steph@math.ubc.ca}
\address{Department of Mathematics,
         University of British Columbia,
         1984 Mathematics Road,
         Vancouver,
         British Columbia,
         V6T 1Z2
         Canada}
\thanks{The third author was supported in part by the National Sciences and Engineering Research Council of Canada.}
\subjclass[2000]{Primary 05E05; 05E10; 14N15} 
\keywords{Jacobi-Trudi determinant, jeu de taquin, ribbon, 
Schubert calculus, 
Schur positive,  skew Schur function, symmetric function} 

\begin{abstract} 
Some new relations on skew Schur function differences are 
established both combinatorially using Sch\"utzenberger's jeu de taquin, and algebraically using Jacobi-Trudi determinants. These relations lead to the conclusion that certain differences
of skew Schur functions are Schur positive. Applying these results 
to a basis of symmetric functions involving ribbon Schur functions confirms the validity of a Schur positivity conjecture due to McNamara. 
A further application reveals that certain differences
of products of Schubert classes are Schubert positive.  
\end{abstract}
\maketitle
\begin{center}\emph{For Manfred Schocker 1970--2006}\end{center}
\section{Introduction}\label{sec:intro} 

Recently there has been much research on Schur positivity of differences of
skew Schur functions, see for example
 \cite{FominFultonLiPoon,  LamPostnikovPylyavskyy, Okounkov}.
In this paper we discover some new Schur positive differences of
skew Schur functions and use them to derive a result on the basis of
symmetric functions consisting of skew Schur functions indexed by certain ribbons,
which is analogous to a well-known result on complete symmetric functions.

This paper is structured as follows.
In the remainder of this section we review the necessary background concerning
tableaux and skew Schur functions.
In Section~\ref{sec:relations} we derive new relations on differences of
skew Schur functions, in particular Lemma~\ref{BasLemSec} and Theorem~\ref{big-difference}. Lemma~\ref{BasLemSec} is the technical heart of the paper and in view of its importance and their independent interest we
offer both
a combinatorial and an algebraic proof. 
The former invokes Sch\"utzenberger's jeu de taquin, while the latter is based on Jacobi-Trudi determinants.
In Section~\ref{sec:positivity} we apply Theorem~\ref{big-difference} to obtain various
new Schur positive differences of skew Schur functions and resolve a
conjecture of McNamara in Theorem~\ref{positive-ribbon-theorem}.
Finally, in Section~\ref{sec:schubert} we show that the preceding results may be used to establish that certain
differences of products of Schubert classes are Schubert positive, in a sense that we define.

\subsection{Partitions and diagrams}\label{sec:partitions_and_diagrams}

Let $\alpha = (\alpha _1,\alpha _2,\ldots ,\alpha _k)$ be a \emph{sequence} of  integers whose sum, $|\alpha|$, is $N$. If the \emph{parts} $\alpha_i$, of $\alpha$, satisfy $\alpha _1,\alpha _2 ,\cdots ,\alpha _{k}>0$ then we say that $\alpha$ is a \emph{composition} of 
$N$, denoted $\alpha \vDash N$. We also say that $\alpha$ has \emph{length} $\ell(\alpha):=k$. If $\alpha _i = \alpha _{i+1} = \cdots = \alpha _{i+j-1} = a$ we normally denote the subsequence $\alpha _i ,\alpha _{i+1}, \ldots , \alpha _{i+j-1}$ by $a^j$. We  denote by $()$ the unique composition of 
$0$. 

For use in some of our proofs later, recall that there exists a bijection between compositions of $N$ and the collection 
$2^{[N-1]}$ of all subsets of $\{1,2,\ldots,N-1\}$ that sends a composition 
$\alpha=(\alpha_1,\ldots,\alpha_{\ell(\alpha)})$ to the set of partial sums
$S(\alpha)=\{\alpha_1,\alpha_1+\alpha_2,\ldots,\alpha_1+\alpha_2+\cdots+\alpha_ {\ell(\alpha)-1} \}$.   If the parts of the composition $\alpha$ satisfy $\alpha _1\geq\alpha _2 \geq\cdots \geq\alpha _{\ell(\alpha)}$ then we say $\alpha$ is a \emph{partition} of $N$, denoted $\alpha \vdash N$. For clarity of exposition we usually denote sequences and compositions by $\alpha, \beta, \sigma , \tau$ and partitions by $\lambda,\mu, \nu$, and  we will use this convention next.

Three partial orders that exist on partitions $\lambda$ and $\mu$ are
\begin{enumerate}
\item the \emph{inclusion order}: $\mu \subseteq \lambda$ if $\mu _i\leq \lambda _i$ for all $1\leq i \leq \ell(\mu)$;
\item the \emph{dominance order} on partitions $\lambda, \mu\vdash N$: $\mu \dom \lambda$ if
$$\mu _1+\cdots +\mu _i \leq \lambda _1+\cdots +\lambda _i $$for all $i$, where if $i>\ell (\lambda)$ (resp. $i>\ell (\mu)$) then $\lambda _i :=0$ (resp. $\mu _i :=0$);
\item the \emph{lexicographic order} on partitions $\lambda, \mu \vdash N$: $\mu \lex \lambda$ if for some $i$ we have $\mu _j = \lambda _j$ for $1\leq j<i$ and $\mu _i < \lambda _i$.
\end{enumerate}

It is not hard to see that the lexicographic order extends the dominance order.  It is also known, see for example \cite{Edelman}, that the cover relations in the dominance order are 
\begin{enumerate}
\item $(\lambda,a,b,\mu)\,\triangleleft_{dom}\,(\lambda,a+1,b-1,\mu)$ for $\lambda_{\ell(\lambda)}>a\geq b>\mu_1$;
\item $(\lambda,a^n,\mu)\,\triangleleft_{dom}\,(\lambda,a+1,a^{n-2},a-1,\mu)$ for
$\lambda_{\ell(\lambda)}>a>\mu_1$ and $n\geq3$.
\end{enumerate} 
For example, $(3,2)\triangleleft _{dom} (4,1)$ and $(2,2,2,2)\triangleleft _{dom} (3,2,2,1)$.

Given a partition $\lambda = (\lambda _1 , \lambda _2 , \ldots , \lambda _{\ell (\lambda )})$ we can associate with it a (\emph{Ferrers} or \emph{Young}) \emph{diagram} also denoted by $\lambda$ that consists of $\lambda _i$ left-justified boxes in row $i$ when read from the top. For ease of referral, boxes will be described by their row and column indices. Furthermore, given two partitions $\lambda , \mu$ such that $\mu\subseteq \lambda$, the \emph{skew diagram} $\lambda / \mu$ is obtained from the diagram $\lambda$ by removing the subdiagram of boxes $\mu$ from the top left corner. In terms of row and column indices
$$\lambda /\mu =\{ (i,j)\, |\, (i,j)\in \lambda, (i,j)\not\in \mu \}.$$
For example, if we denote boxes by $\times$ then
$$(3,3,2,1)/(1,1)=\begin{matrix}
&\times&\times\\
&\times&\times\\
\times&\times\\
\times \end{matrix}\,.$$ Additionally, two skew diagrams will be considered equivalent if one can be obtained from the other by the removal  of empty rows or empty columns. We say a skew diagram is \emph{connected} if it is edgewise connected.

We can describe skew diagrams $\lambda /\mu$ a third way using $1$- and $2$-row overlap 
sequences as defined in \cite{HDL2}. Let $\lambda , \mu$ be diagrams such that $\lambda /\mu$ is a 
skew diagram occupying $\ell (\lambda)$ rows. Then the \emph{$1$-row overlap 
sequence} of $\lambda /\mu$ is
$$r^{(1)}(\lambda /\mu) = (\lambda _1 - \mu _1, \lambda _2 - \mu _2, \ldots , \lambda _{\ell(\mu)}-\mu _{\ell(\mu)}, \lambda _{\ell(\mu)+1}, \ldots ,\lambda _{\ell(\lambda)})$$and the \emph{$2$-row overlap 
sequence} of $\lambda /\mu$ is
$$r^{(2)}(\lambda /\mu) = (\lambda _2 - \mu _1, \lambda _3 - \mu _2, \ldots , \lambda _{\ell(\mu)+1}-\mu _{\ell(\mu)}, \lambda _{\ell(\mu)+2}, \ldots ,\lambda _{\ell(\lambda)}).$$Hence, considering our previous skew diagram $$r^{(1)}((3,3,2,1)/(1,1))=(2,2,2,1), \quad r^{(2)} ((3,3,2,1)/(1,1)) = (2,1,1).$$ Note that this completely describes $\lambda /\mu$ and so we can also denote the skew diagram by 
\begin{equation}
\label{Eq-skew-overlap}
\lambda/\mu=(r^{(1)}(\lambda/\mu)\, |\, r^{(2)}(\lambda/\mu)).
\end{equation}
If $r^{(2)}(\lambda/\mu)=(1^k)$ with  $k=\ell(\lambda)-1$ then we say $\lambda /\mu$ is a \emph{ribbon} (or \emph{border strip} or \emph{rim hook}) and since $r^{(2)}(\lambda/\mu)$ is predetermined we can describe the ribbon completely by the composition $r^{(1)}(\lambda /\mu)$.

One last notion that we need for  diagrams is that of inner and outer corners. If $\lambda$ is a diagram, then we say $(i,j)$ is an \emph{inner corner} of $\lambda$ if $(i,j)\in \lambda$ and $\lambda $ with the box in position $(i,j)$ removed is also a diagram. Similarly we say $(i,j)$ is an \emph{outer corner} of $\lambda$ if $(i,j)\not\in \lambda$ and $\lambda $ with a box appended in position $(i,j)$ is also a diagram.
For example, if $\lambda = (3,3,2,1)$ then the inner corners of $\lambda$ are denoted by $\otimes$ and the outer corners by $\star$:
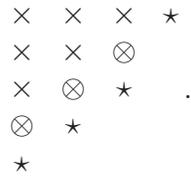
$$\begin{matrix}
\times&\times&\times&\star\\
\times&\times&\otimes\\
\times&\otimes&\star\\
\otimes&\star\\
\star\end{matrix}\,.$$Inner and outer corners will play a vital role in the next subsection.

\subsection{Tableaux and jeu de taquin}

Consider a skew diagram $\lambda /\mu$. We say we have a \emph{tableau}, $T$, of \emph{shape} $\lambda /\mu$ if each box is filled with a positive integer. In addition, we say we have a \emph{semistandard Young tableau} ({SSYT}, plural {SSYTx}) if
\begin{enumerate}
\item As we read the entries in each row of $T$ from left to right the entries weakly increase;
\item As we read the entries in each column of $T$ from top to bottom the entries strictly increase.
\end{enumerate}
If we read the entries of $T$ from right to left and top to bottom then the resulting word is called the \emph{reading word} of $T$, $w(T)$. If, for each $i$, the number of $i\,$s we have read is always at least the number of $(i+1)\,$s we have read then we say $w(T)$ is \emph{lattice}. Let $c_i(T)$ be the total number of $i\,$s appearing in $T$ and also in $w(T)$, then the list $c(T):=(c_1(T),c_2(T),\ldots)$ is known as the \emph{content} of $T$ and also of $w(T)$. If $T$ is a SSYT with $n$ entries and $c(T)=(1^n)$  then we say $T$ is a \emph{standard Young tableau} ({SYT}, plural {SYTx}). 

Lastly, if we have a SSYT we are able to perform  Sch\"utzenberger's \emph{jeu de taquin} (\emph{jdt}), see for example \cite{Sagan}, on it.

\begin{definition}
Given a SSYT, $T$, of shape $\lambda /\mu$ we perform a \emph{forward-(jdt)-slide} as follows:
\begin{enumerate}
\item Choose an inner corner of $\mu$, $C=(i,j)$.
\item Let $c=\min \{T(i+1, j), T(i, j+1)\}$ or $c=T(i+1, j)$ if $T(i+1, j) = T(i, j+1)$  or if only one of $T(i+1, j), T(i, j+1)$ exists then that is taken to be the minimum value. Let $C'$ be the position $c$ is in.
\item Form $T'$ by setting $T(i,j)=c$ and letting $C'$ be empty.
\item Set $C:=C'$ and return to the second step until $T'$ is a SSYT.
\end{enumerate}

Given a SSYT, $T$, of shape $\lambda /\mu$ we perform a \emph{backward-(jdt)-slide} as follows:
\begin{enumerate}
\item Choose an outer corner of $\lambda$, $D=(i,j)$.
\item Let $d=\max \{T(i-1, j), T(i, j-1)\}$ or $d=T(i-1, j)$ if $T(i-1, j) = T(i, j-1)$  or if only one of $T(i-1, j), T(i, j-1)$ exists then that is taken to be the maximum value. Let $D'$ be the position $d$ is in.
\item Form $T''$ by setting $T(i,j)=d$ and letting $D'$ be empty.
\item Set $D:=D'$ and return to the second step until $T''$ is a SSYT.
\end{enumerate}
\end{definition}

Note the output of each algorithm is a SSYT and that these slides are invertible, as illustrated in the following example.

\begin{example}
Let $T=\begin{smallmatrix}&&1&1&2\\&1&3\\2&4\end{smallmatrix}$ and $C=(1,2)$ then a forward-jdt-slide takes place as shown, where $\bullet$ indicates the position of $C$ at each stage.
$$\begin{matrix}&\bullet &1&1&2\\&1&3\\2&4\end{matrix} \longrightarrow \begin{matrix}&1 &1&1&2\\&\bullet&3\\2&4\end{matrix} \longrightarrow \begin{matrix}&1 &1&1&2\\&3&\bullet\\2&4\end{matrix}. $$

Conversely, let $U=\begin{smallmatrix}&1 &1&1&2\\&3&\\2&4\end{smallmatrix}$ and $D=(2,3)$ then a backward-jdt-slide takes place as shown, where $\bullet$ indicates the position of $D$ at each stage.
$$\begin{matrix}&1 &1&1&2\\&3&\bullet\\2&4\end{matrix}\longrightarrow \begin{matrix}&1 &1&1&2\\&\bullet&3\\2&4\end{matrix} \longrightarrow \begin{matrix}&\bullet &1&1&2\\&1&3\\2&4\end{matrix}. $$
\end{example}

We are now ready to introduce skew Schur functions.

\subsection{The algebra of symmetric functions}

For any set ${\mathcal T}$ of tableaux, we define the
generating function
\begin{equation}\label{tableaux-gen}
g({\mathcal T})=\sum_{T\in\mathcal T} x^{c(T)}
\end{equation}where $x^{c(T)}:= x_1^{c_1(T)}x_2^{c_2(T)}\cdots .$  This generating function can be used to define a basis of the algebra of symmetric functions, $\Lambda$, known as the basis of \emph{Schur functions} $\{s_\lambda \} _{\lambda \vdash N}$ through
\begin{equation}\label{schur-def-monomial}
s_\lambda :=g({\mathcal T}^\lambda)
\end{equation}where ${\mathcal T}^\lambda$ is the set of all SSYTx of shape $\lambda$. Another basis of the algebra of symmetric functions is the basis of \emph{complete symmetric functions} $\{h_\lambda \} _{\lambda \vdash N}$ where $h_0 :=1$,
$$h _\lambda := h _{\lambda _1}h _{\lambda _2}\cdots h _{\lambda _{\ell(\lambda)}}$$and
$$h _k := \sum _{i_1\leq i_2\leq \cdots \leq i_k}x _{i _1}x _{i _2}\cdots x _{i_k}.$$Either of these bases can be used to describe our desired objects of study, skew Schur functions.

Given a skew diagram $\lambda /\mu$ we define the \emph{skew Schur function} by 
\begin{equation}\label{skew-schur-def-monomial}
s_{\lambda /\mu}:=g({\mathcal T}^{\lambda/\mu})
\end{equation}where ${\mathcal T}^{\lambda/\mu}$ is the set of all SSYTx of shape $\lambda /\mu$. For clarity of exposition we will use the less conventional \emph{overlap} notation
\begin{equation}
\label{Eq-sfn-overlap}
s_{\lambda/\mu}=\{\lambda /\mu\}=\{r^{(1)}(\lambda/\mu)\, |\, r^{(2)}(\lambda/\mu)\}
\end{equation}
to state our results from the next section onwards. In terms of complete symmetric functions
\begin{equation}\label{skew-schur-def-h}
s _{\lambda /\mu}=
 \left|\ h_{\lambda _i -\mu _j - i +j}\ \right|_{i,j=1}^{\ell(\lambda)}
\end{equation}
where $\mu _j:=0$ if $j>\ell (\mu)$ and $h_k:=0$ if $k<0$,
see for example \cite{ECII}.
These determinants are known as \emph{Jacobi-Trudi determinants}.
Expanding, instead, in the Schur function basis we have
\begin{equation}\label{skew-schur-def-schur}  s _{\lambda /\mu}=\sum _\nu c^\lambda _{\mu\nu} s_\nu\end{equation}where $c^\lambda _{\mu\nu}$ is the \emph{Littlewood-Richardson coefficient} defined to be the number of SSYTx, $T$, of shape $\lambda /\mu$ where $w(T)$ is lattice and $c(T)=\nu$. This manner of determining 
$c^\lambda _{\mu \nu}$ is called the \emph{Littlewood-Richardson rule}, and clearly yields that $c^\lambda _{\mu \nu}$ is a non-negative integer.

In what follows, we are especially interested in the case where $\lambda /\mu$
is a ribbon.
In such a case, $\lambda/\mu=(\alpha\,|\,1^{\ell(\lambda)-1})$
where $\alpha=r^{(1)}(\lambda/\mu)$,
and we define $r_\alpha=s_{\lambda/\mu}\equiv\{\alpha\,|\,1^{\ell(\lambda)-1}\}$.
We call $r_{\alpha}$ a \emph{ribbon Schur function}.
The set
$\{r_\lambda\}_{\lambda\vdash N}$ forms another basis for the algebra of symmetric
functions \cite[Section 2.1]{HDL}. Ribbon Schur functions are also significant because they are, for example, useful in computing the number of permutations with 
a given cycle structure and descent set \cite{GR}, and they can be used to compute skew Schur 
functions via determinants consisting of  associated ribbon Schur functions \cite{HamelGoulden}.

 
Note that the Jacobi-Trudi determinant ~\eqref{skew-schur-def-h} for the ribbon
Schur
function
$$r_\alpha=\{\alpha_1,\ldots,\alpha_{\ell(\alpha)}\,|\,1^{\ell(\alpha)-1}\}$$takes the form
\begin{equation}
   r_\alpha= \left|\ h_{R_{ij}}\ \right|_{i,j=1}^{\ell(\alpha)}\,
\end{equation}
with
\begin{equation}
   R_{ij}=\left\{
\begin{array}{ll}
 \alpha_i+\alpha_{i+1}+\cdots+\alpha_{j} &\text{if $i\leq j$;}\\
 0 &\text{if $i=j+1$;}\\
 -1 &\text{if $i>j+1$,}\\
\end{array} \right. \,
\end{equation}
so that
\begin{equation}
   r_\alpha=\left|
   \begin{array}{lllll}
   h_{\alpha_1}&h_{\alpha_1+\alpha_2}&h_{\alpha_1+\alpha_2+\alpha_3}&\cdots \\ \\
   1&h_{\alpha_2}&h_{\alpha_2+\alpha_3}&\cdots \\ \\
   0&1&h_{\alpha_3}&\cdots\\ \\
   \vdots&\vdots&\vdots&\ddots\\
   \end{array}
      \right|\,
\end{equation}
where we have substituted $h_0=1$ and $h_k=0$ for $k<0$.

More generally, for $$s_{\lambda/\mu}=\{\alpha_1,\ldots,\alpha_{\ell(\alpha)}\,|\,\beta_1+1,\ldots,\beta_{\ell(\alpha)-1}+1\}$$we have
\begin{equation}
   s_{\lambda/\mu}= \left|\ h_{Q_{ij}}\ \right|_{i,j=1}^{\ell(\alpha)}\,
\end{equation}
with
\begin{equation}\label{GeneralMatrix}
   Q_{ij}=\left\{
\begin{array}{ll}
 (\alpha_i+\alpha_{i+1}+\cdots+\alpha_{j})-(\beta_i+\beta_{i+1}+\cdots+\beta_{j-1}) &\text{if $i\leq j-1$;}\\
 \  \alpha_i &\text{if $i=j$;}\\
 \  \beta_j  &\text{if $i=j+1$;}\\
 -(\alpha_{j+1}+\alpha_{j+2}+\cdots+\alpha_{i-1})+(\beta_j+\beta_{j+1}+\cdots+\beta_{i-1})&\text{if $i> j+1$,}\\
 \end{array} \right. \,
\end{equation}
so that
\begin{equation}\label{GeneralJ-T}
   s_{\lambda/\mu}=\left|
   \begin{array}{lllll}
   h_{\alpha_1}&h_{\alpha_1+\alpha_2-\beta_1}&h_{\alpha_1+\alpha_2+\alpha_3-\beta_1-\beta_2}&\cdots \\ \\
   h_{\beta_1}&h_{\alpha_2}&h_{\alpha_2+\alpha_3-\beta_2}&\cdots \\ \\
   h_{\beta_1+\beta_2-\alpha_2}&h_{\beta_2}&h_{\alpha_3}&\cdots\\ \\             
   h_{\beta_1+\beta_2+\beta_3-\alpha_2-\alpha_3}&h_{\beta_2+\beta_3-\alpha_3}&h_{\beta_3}&\cdots\\
   \vdots&\vdots&\vdots&\ddots\\
   \end{array}
      \right| \,.
   \end{equation}
 
Note that \eqref{skew-schur-def-schur} 
is a non-negative linear combination of Schur functions. This motivates the following definition.

\begin{definition} If a symmetric function $f\in \Lambda$ can be written as a non-negative linear combination of Schur functions then we say that $f$ is \emph{Schur positive}.
\end{definition}

Our goal for the remainder of this paper is to construct new Schur positive expressions. We end our introduction with a classical Schur positive linear combination that will serve as a motivation for some of our later results.

\begin{theorem}\cite[p119]{Macdonald}\label{h-schur-positive} Let $\lambda , \mu \vdash N$ then 
$$h _\mu - h _\lambda$$is Schur positive if and only if $\mu  \dom\lambda$.
\end{theorem}

\section{Skew Schur function differences}\label{sec:relations}

In this section we study a variety of differences of skew Schur functions, and moreover discover that some of them are Schur positive. Before we proceed with these differences, we introduce the following hypothesis.

\begin{hypothesis}\label{ovsigma} 
Let $\sigma$ and $\tau$ be compositions such that
 $\ell(\sigma)=s\geq 0$ and $\ell(\tau)=t\geq 0$,
 and let $\ov\sigma$ and $\ov\tau$ be sequences of
 non-negative integers
 that satisfy the following conditions:
\begin{enumerate}
\item The lengths of $\ov\sigma$ and $\ov\tau$ are $s$ and $t$ respectively;
\item $\ov\sigma _s=1$ when $s>0$;
\item $\ov\tau _1=1$ when $t>0$;
\item $\bar\sigma_i\le\min\{\sigma_i,\sigma_{i+1}\}$ for $1\le i<s$, and
$\bar\tau_i\le\min\{\tau_i,\tau_{i-1}\}$ for $1< i\le t$.
\end{enumerate}
\end{hypothesis}

\subsection{A combinatorial approach}

\begin{lemma}\label{BasLemSec}

Assume $\sigma ,\tau ,\ov\sigma ,\ov\tau$ satisfy Hypothesis~\ref{ovsigma}. 
If $m\ge1$, $n\ge2$ and $0\le x\le\min\{m,n-1\}$ then
\begin{equation}\label{BasicEq}
\begin{split}
&\{\sigma,m,n,\tau\,|\,\bar\sigma,x,\bar\tau\}
-\{\sigma,m+1,n-1,\tau\,|\,\bar\sigma,x,\bar\tau\}\\
&\quad=\quad
\begin{cases}
\{\sigma,m,n,\tau\,|\,\bar\sigma,x+1,\bar\tau\}
  &\text{if $x<m$;}\\
-\{\sigma\backslash\sigma_s,m+\sigma_s,n,\tau\,|\,\bar\sigma\backslash\bar\sigma_s,m+1,\bar\tau\}
  &\text{if $x=m$,}
\end{cases}
\\
&\qquad\qquad-
\begin{cases}
\{\sigma,m+1,n-1,\tau\,|\,\bar\sigma,x+1,\bar\tau\}
  &\text{if $x<n-1$;}\\
-\{\sigma,m+1,n-1+\tau_1,\tau\backslash\tau_1\,|\,\bar\sigma,n,\bar\tau\backslash\bar\tau_1\}
  &\text{if $x=n-1$.}
\end{cases}
\end{split}
\end{equation}
If $s=0$ then the term containing
$\sigma\backslash\sigma_s$ is to be omitted.
Similarly, if $t=0$ then the term
containing $\tau\backslash\tau_1$ is to be omitted.\end{lemma}


\begin{proof} 
\newcommand{\tabset}[1]{{\mathcal T}^{#1}}
\newcommand{\forwardA}{\xi}
\newcommand{\backwardA}{\zeta}
\newcommand{\forwardC}{\eta}
\newcommand{\backwardC}{\gamma}



In this proof, we use Sch\"utzenberger's jeu de taquin
to
 construct invertible maps between various sets of
tableaux of
 skew shape, and then use the corresponding generating function \eqref{tableaux-gen}, to obtain \eqref{BasicEq} by means of \eqref{skew-schur-def-monomial} and \eqref{Eq-sfn-overlap}. 
 The proof runs through the four cases demarcated by
~\eqref{BasicEq}.
 First we prove the case where $x<m$ and $x<n-1$.
 The other three cases are then treated as variants of
this case. 

Before we begin, we introduce some notation.
For any  skew diagram $\kappa=(\alpha\, |\, \beta)$, let
$\tabset{\kappa}$ denote the set of all SSYTx of shape $\kappa$,
so that $\{\kappa\}=g(\tabset{\kappa})$.

Let $$\forwardA=(\sigma,m,n,\tau\,|\,\bar\sigma,x,\bar\tau)$$
$$\backwardA=(\sigma,m+1,n-1,\tau\,|\,\bar\sigma,x,\bar\tau)$$
and let $T\in\tabset{\forwardA}$.
For the purposes of a forward-jdt-slide, let $C$ be the vacancy
immediately to the left of the first entry in the $(s+1)$th row of $T$.

We now consider our first case, for which $x<m$ and $x<n-1$.
The latter of these constraints implies that the vacancy $C$
has just one node from $T$ below it.
Thus, forward-sliding $C$ through $T$ necessarily results in this
vacancy migrating to the end of either
the $(s+1)$th or $(s+2)$th row,
as depicted in Figs.~\ref{ShiftU1Eq} and \ref{ShiftU2Eq} respectively.


\begin{figure}[ht]
\begin{equation*}
{\setlength{\arraycolsep}{2pt}
\begin{matrix}
&&&&&&&&&&\sigma&\sigma&\sigma\\
&&&&\bullet&\times&\cdots&\times&\times&\times&\times\\
&&\times&\times&\times&\times&\cdots&\times&\times\\
\tau&\tau&\tau
\end{matrix}
\longrightarrow
\begin{matrix}
&&&&&&&&&&\sigma&\sigma&\sigma\\
&&&&\times&\times&\cdots&\times&\times&\times&\bullet\\
&&\times&\times&\times&\times&\cdots&\times&\times\\
\tau&\tau&\tau
\end{matrix}}
\end{equation*}
\caption{Map from $\tabset{\forwardA}$ to
                  ${\mathcal U}_1\subset\tabset{\forwardA^*}$ via a forward-slide}
\label{ShiftU1Eq}
\end{figure}

\begin{figure}[ht]
\begin{equation*}
{\setlength{\arraycolsep}{2pt}
\begin{matrix}
&&&&&&&&&&\sigma&\sigma&\sigma\\
&&&&\bullet&\times&\cdots&\times&\times&\times&\times\\
&&\times&\times&\times&\times&\cdots&\times&\times\\
\tau&\tau&\tau
\end{matrix}
\longrightarrow
\begin{matrix}
&&&&&&&&&&\sigma&\sigma&\sigma\\
&&&&\times&\times&\cdots&\times&\times&\times&\times\\
&&\times&\times&\times&\times&\cdots&\times&\bullet\\
\tau&\tau&\tau
\end{matrix}}
\end{equation*}
\caption{Map from $\tabset{\forwardA}$ to
                  ${\mathcal U}_2\subset\tabset{\backwardA}$ via a forward-slide}
\label{ShiftU2Eq}
\end{figure}
In these two cases, the resulting SSYT is of shape
$\forwardA^*=(\sigma,m,n,\tau\,|\,\bar\sigma\backslash\bar\sigma_s,0,x+1,\bar\tau)$
or shape $\backwardA$ respectively.
Let ${\mathcal U}_1$ and ${\mathcal U}_2$ be the sets of
SSYTx of shapes $\forwardA^*$ and $\backwardA$ respectively that
arise from performing a forward-slide as above on all the
elements of $\tabset{\forwardA}$.
Then $\{\forwardA\}=g({\mathcal U}_1)+g({\mathcal U}_2)$.

However, neither
${\mathcal U}_1$ nor ${\mathcal U}_2$ is the full
set of SSYTx of shape $\forwardA^*$ or $\backwardA$ respectively.
In particular, for each tableau $U\in{\mathcal U}_1$,
the entry at the end of the $(s+1)$th row
(immediately to the left of the vacancy)
is larger than that at the beginning of the $s$th row
(immediately above the vacancy),
because in the preimage $T$ of $U$, the former of these entries
would have been immediately below the latter.
We claim that all SSYTx of shape $\forwardA^*$ that satisfy this
constraint occur in ${\mathcal U}_1$.
To see this, let $U$ be an arbitrary such SSYT.
Performing a backward-jdt-slide on $U$ necessarily
results in a SSYT of shape $\forwardA$, which is thus an element of
$\tabset{\forwardA}$.
That forward-sliding is the inverse of backward-sliding then guarantees
that $U\in{\mathcal U}_1$.


Now form the set ${\mathcal U}_1^<$ of tableaux by,
for each tableau $U\in{\mathcal U}_1$, shifting each entry
in the first $s$ rows of $U$ one position to its left.
This shift is indicated in Fig.~\ref{ShiftU1bEq}.
\begin{figure}[ht]
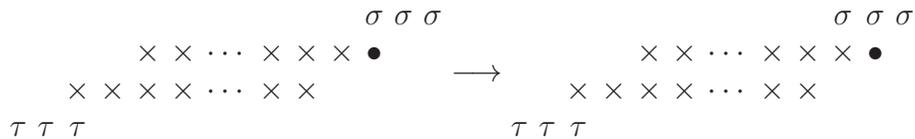

\begin{equation*}
{\setlength{\arraycolsep}{2pt}
\begin{matrix}
&&&&&&&&&&\sigma&\sigma&\sigma\\
&&&&\times&\times&\cdots&\times&\times&\times&\bullet\\
&&\times&\times&\times&\times&\cdots&\times&\times\\
\tau&\tau&\tau
\end{matrix}
\longrightarrow
\begin{matrix}
&&&&&&&&&\sigma&\sigma&\sigma\\
&&&&\times&\times&\cdots&\times&\times&\times&\bullet\\
&&\times&\times&\times&\times&\cdots&\times&\times\\
\tau&\tau&\tau
\end{matrix}}
\end{equation*}
\caption{Map from ${\mathcal U}_1\subset\tabset{\forwardA^*}$ to
                  ${\mathcal U}_1^<\subset\tabset{\forwardA^+}$
                  via left-shifting the entries $\sigma$}
\label{ShiftU1bEq}
\end{figure}
The elements of ${\mathcal U}_1^<$ are of shape
$$\forwardA^+=(\sigma,m,n,\tau\,|\,\bar\sigma,x+1,\bar\tau).$$
The above constraint on the entries of each element of
${\mathcal U}_1$ ensures that each element of
${\mathcal U}_1^<$ is a SSYT.
Moreover, by reversing the shift, we see that
${\mathcal U}_1^<=\tabset{\forwardA^+}$.
Therefore, $\tabset{\forwardA^+}$ and ${\mathcal U}_1$ are in bijection,
and thus $\{\forwardA^+\}=g({\mathcal U}_1)$.
Consequently, $\{\forwardA\}-\{\forwardA^+\}=g({\mathcal U}_2)$.

Now consider $T\in\tabset{\backwardA}$.
For the purposes of a backward-slide, let $D$ be the vacancy immediately
to the right of the final entry in the $(s+2)$th row of $T$.
The constraint $x<m$ implies that the vacancy $D$
has just one node from $T$ above it.
Thus, backward-sliding $D$ through $T$ necessarily results in this
vacancy migrating to the beginning of either
the $(s+2)$th or $(s+1)$th row.
These two instances are
depicted in Figs.~\ref{ShiftV1Eq} and \ref{ShiftV2Eq} respectively.
\begin{figure}[ht]
\begin{equation*}
{\setlength{\arraycolsep}{2pt}
\begin{matrix}
&&&&&&&&&&\sigma&\sigma&\sigma\\
&&&&\times&\times&\cdots&\times&\times&\times&\times\\
&&\times&\times&\times&\times&\cdots&\times&\bullet\\
\tau&\tau&\tau
\end{matrix}
\longrightarrow
\begin{matrix}
&&&&&&&&&&\sigma&\sigma&\sigma\\
&&&&\times&\times&\cdots&\times&\times&\times&\times\\
&&\bullet&\times&\times&\times&\cdots&\times&\times\\
\tau&\tau&\tau
\end{matrix}}
\end{equation*}
\caption{Map from $\tabset{\backwardA}$ to
                  ${\mathcal V}_1\subset\tabset{\backwardA^*}$ via a backward-slide}
\label{ShiftV1Eq}
\end{figure}
\begin{figure}[ht]
\begin{equation*}
{\setlength{\arraycolsep}{2pt}
\begin{matrix}
&&&&&&&&&&\sigma&\sigma&\sigma\\
&&&&\times&\times&\cdots&\times&\times&\times&\times\\
&&\times&\times&\times&\times&\cdots&\times&\bullet\\
\tau&\tau&\tau
\end{matrix}
\longrightarrow
\begin{matrix}
&&&&&&&&&&\sigma&\sigma&\sigma\\
&&&&\bullet&\times&\cdots&\times&\times&\times&\times\\
&&\times&\times&\times&\times&\cdots&\times&\times\\
\tau&\tau&\tau
\end{matrix}}
\end{equation*}
\caption{Map from $\tabset{\backwardA}$ to
                  ${\mathcal V}_2\subset\tabset{\forwardA}$ via a backward-slide}
\label{ShiftV2Eq}
\end{figure}
In these two cases, the resulting SSYT is of shape
$\backwardA^*=(\sigma,m+1,n-1,\tau\,|\,\bar\sigma,x+1,0,\bar\tau\backslash\bar\tau_1)$
or shape $\forwardA$ respectively.
Let ${\mathcal V}_1$ and ${\mathcal V}_2$ be the sets of
SSYTx of shapes $\backwardA^*$ and $\forwardA$ respectively
that result from performing a backward-slide as above on all
the elements of $\tabset{\backwardA}$.
Then $\{\backwardA\}=g({\mathcal V}_1)+g({\mathcal V}_2)$.

In analogy with the treatment of ${\mathcal U}_1$ above,
for each element $V\in{\mathcal V}_1$ 
we shift each of the entries in rows $s+3$ and below one position to the right,
as indicated in Fig.~\ref{ShiftV1b}.
\begin{figure}[ht]
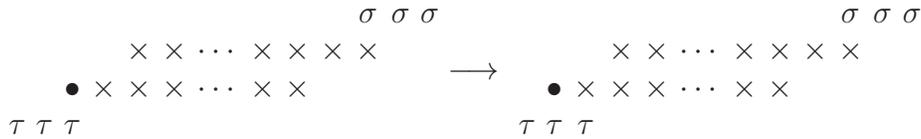

\begin{equation*}
{\setlength{\arraycolsep}{2pt}
\begin{matrix}
&&&&&&&&&&\sigma&\sigma&\sigma\\
&&&&\times&\times&\cdots&\times&\times&\times&\times\\
&&\bullet&\times&\times&\times&\cdots&\times&\times\\
\tau&\tau&\tau
\end{matrix}
\longrightarrow
\begin{matrix}
&&&&&&&&&&\sigma&\sigma&\sigma\\
&&&&\times&\times&\cdots&\times&\times&\times&\times\\
&&\bullet&\times&\times&\times&\cdots&\times&\times\\
&\tau&\tau&\tau
\end{matrix}
}
\end{equation*}
\caption{Map from ${\mathcal V}_1\subset\tabset{\backwardA^*}$ to
                  ${\mathcal V}_1^>\subset\tabset{\backwardA^+}$
                  via right-shifting the entries $\tau$}
\label{ShiftV1b}
\end{figure}
This forms a set ${\mathcal V}_1^>$ of SSYTx of shape
$$\backwardA^+=(\sigma,m+1,n-1,\tau\,|\,\bar\sigma,x+1,\bar\tau).$$
This yields $\{\backwardA^+\}=g({\mathcal V}_1)$, and
consequently, $\{\backwardA\}-\{\backwardA^+\}=g({\mathcal V}_2).$

Now, backward-sliding maps the elements of
${\mathcal U}_2\subset\tabset{\backwardA}$
bijectively onto a subset of ${\mathcal V}_2$ because
backward-sliding is inverse to forward-sliding.
Similarly, forward-sliding maps the elements of
${\mathcal V}_2\subset\tabset{\forwardA}$ bijectively
onto a subset of ${\mathcal U}_2$ because
forward-sliding is inverse to backward-sliding.
It follows that ${\mathcal U}_2$ and ${\mathcal V}_2$ are in bijection,
and thus $g({\mathcal U}_2)=g({\mathcal V}_2)$.
Consequently, $\{\forwardA\}-\{\forwardA^+\}=\{\backwardA\}-\{\backwardA^+\}$,
which yields the case of the lemma for which $x<m$ and $x<n-1$.


We now consider the case in which $x<m$ and $x=n-1$.
In this case, for $T\in\tabset{\forwardA}$, the vacancy $C$ has
more than one entry from $T$ below it
(assume for now that $\tau$ is not empty).
Then, on performing a forward-slide, the vacancy migrates to one
of three positions.
The first two are, as in the first case above, at the ends of
rows $s+1$ and $s+2$.
The third position is at the bottom of the column
that intially contained $C$.
This instance is depicted in Fig.~\ref{ShiftU3Eq}.
\begin{figure}[ht]
\begin{equation*}
{\setlength{\arraycolsep}{2pt}
\begin{matrix}
\phantom{\times}&\phantom{\times}&\phantom{\times}&&&&&&&&&
\sigma&\sigma&\sigma\\
&&&\bullet&\times&\times&\times&\cdots&\times&\times&\times&\times\\
&&&\times&\times&\times&\times&\cdots&\times&\times\\
&&\tau&\tau\\
&\tau&\tau&\tau\\
&\tau&\tau&\tau\\
\tau&\tau
\end{matrix}
\longrightarrow
\begin{matrix}
\phantom{\times}&\phantom{\times}&\phantom{\times}&&&&&&&&&
\sigma&\sigma&\sigma\\
&&&\times&\times&\times&\times&\cdots&\times&\times&\times&\times\\
&&&\tau&\times&\times&\times&\cdots&\times&\times\\
&&\tau&\tau\\
&\tau&\tau&\tau\\
&\tau&\tau&\bullet\\
\tau&\tau
\end{matrix}}
\end{equation*}
\caption{Map from $\tabset{\forwardA}$ to
                  ${\mathcal U}_3$ via a forward-slide}
\label{ShiftU3Eq}
\end{figure}
Let ${\mathcal U}_3$ be the set of all SSYTx that result in this third case.
As above,
${\mathcal U}_1$ and ${\mathcal U}_2$ are defined to be the sets of
SSYTx of shapes $\forwardA^*$ and $\backwardA$ respectively that result from
the first two cases.
Then $\{\forwardA\}=g({\mathcal U}_1)+g({\mathcal U}_2)+g({\mathcal U}_3)$.
Moreover, we obtain $\{\forwardA^+\}=g({\mathcal U}_1)$ as in the
first case, resulting in
$\{\forwardA\}-\{\forwardA^+\}-g({\mathcal U}_3)=g({\mathcal U}_2)$.

Note that, for $U\in{\mathcal U}_3$,
each entry in the same column of $U$ as the vacancy
is greater than or equal to the entry (if there is one)
below and to its left.
This follows because $U$ was obtained from a SSYT
$T\in\tabset{\forwardA}$ by a sequence of slides that move downwards.

Now form the set ${\mathcal U}_3^\wedge$ by, for each
$U\in{\mathcal U}_3$, shifting up one position each entry
in all the columns to the left of that of the vacancy.
This shift is depicted in Fig.~\ref{ShiftU3UEq}.
\begin{figure}[ht]
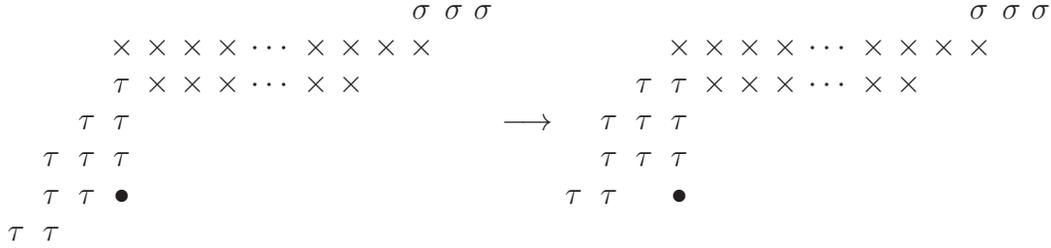

\begin{equation*}
{\setlength{\arraycolsep}{2pt}
\begin{matrix}
\phantom{\times}&\phantom{\times}&\phantom{\times}&&&&&&&&&
\sigma&\sigma&\sigma\\
&&&\times&\times&\times&\times&\cdots&\times&\times&\times&\times\\
&&&\tau&\times&\times&\times&\cdots&\times&\times\\
&&\tau&\tau\\
&\tau&\tau&\tau\\
&\tau&\tau&\bullet\\
\tau&\tau
\end{matrix}
\longrightarrow
\begin{matrix}
\phantom{\times}&\phantom{\times}&\phantom{\times}&&&&&&&&&
\sigma&\sigma&\sigma\\
&&&\times&\times&\times&\times&\cdots&\times&\times&\times&\times\\
&&\tau&\tau&\times&\times&\times&\cdots&\times&\times\\
&\tau&\tau&\tau\\
&\tau&\tau&\tau\\
\tau&\tau&&\bullet\\
\strut
\end{matrix}}
\end{equation*}
\caption{Map from ${\mathcal U}_3$ to
                  ${\mathcal U}_3^\wedge\subset\tabset{\forwardC}$
                  via up-shifting of some entries $\tau$}
\label{ShiftU3UEq}
\end{figure}
Each of the resulting elements of ${\mathcal U}_3^\wedge$ is of shape
$$\forwardC=(\sigma,m+1,n-1+\tau_1,\tau\backslash\tau_1\,|\,\bar\sigma,x+1,\bar\tau\backslash\bar\tau_1).$$
In view of the above note, each of these elements is a SSYT.
Indeed, ${\mathcal U}_3^\wedge$ is in bijection with
$\tabset{\forwardC}$, as may be seen by, in each element of the latter,
shifting downward each of the entries in the columns to the
left of that of the vacancy and then performing a backward-slide.
Therefore, $\{\forwardC\}=g({\mathcal U}_3^\wedge)=g({\mathcal U}_3)$.
Then, from above,
$\{\forwardA\}-\{\forwardA^+\}-\{\forwardC\}=g({\mathcal U}_2)$.
When $\tau$ is empty, the sets
$g({\mathcal U}_3)$ and $g({\mathcal U}_3^\wedge)$ are empty,
and the correct expression is given upon setting $\{\forwardC\}=0$.

Now, as in the first case, consider $T\in\tabset{\backwardA}$, and
for the purposes of a backward-slide, let $D$ be the vacancy immediately
to the right of the final entry in the $(s+2)$th row of $T$.
However, $x=n-1$, here, implies that the $(s+1)$th and $(s+2)$th rows
of $T$ are flush at their left edges, and therefore
terms depicted in Fig. \ref{ShiftV1Eq} do not arise.
Consequently, $\{\backwardA\}=g({\mathcal V}_2)$ where,
as in the first case, ${\mathcal V}_2$ is the set of
SSYTx of shape $\forwardA$ that results from performing a
backward-slide on all the elements of $\tabset{\backwardA}$.

As in the first case, ${\mathcal U}_2$ and ${\mathcal V}_2$
are in bijection, whence $g({\mathcal U}_2)=g({\mathcal V}_2)$.
Thereupon, $\{\forwardA\}-\{\forwardA^+\}-\{\forwardC\}=\{\backwardA\}$,
and the required expression in the case for
which $x<m$ and $x=n-1$ follows.


The case for which $x=m$ and $x<n-1$ is similar to the case just
considered, but rotated $180^\circ$.
Here, rows $(s+1)$ and $(s+2)$ of $T\in\tabset{\forwardA}$ are
flush at their right ends.
Consequently, forward-sliding $C$ through $T$
necessarily results in this vacancy migrating to the end of
the $(s+2)$th row, as in Fig.~\ref{ShiftU2Eq}.
The resulting SSYT is then necessarily of shape $\backwardA$.
With the set ${\mathcal U}_2$ formed by enacting this
forward-slide on all elements of $\tabset{\forwardA}$,
we obtain $\{\forwardA\}=g({\mathcal U}_2)$.

Now, as in the previous cases, consider $T\in\tabset{\backwardA}$, and
let $D$ be the vacancy immediately
to the right of the final entry in the $(s+2)$th row of $T$.
In this instance, though, $D$ has more than one entry of $T$
above it (provided that $\sigma$ is not empty).
Thus in addition to the two terms arising from the migration
of $D$ to the beginnning of rows $(s+1)$ and $(s+2)$,
as in Fig.~\ref{ShiftV1Eq} and Fig.~\ref{ShiftV2Eq} respectively,
there is a term arising from the migration
of $D$ directly upwards, as in Fig.~\ref{ShiftV3Eq}.
\begin{figure}[ht]
\begin{equation*}
{\setlength{\arraycolsep}{2pt}
\begin{matrix}
&&&&&&&&&&&&\sigma\\
&&&&&&&&&\sigma&\sigma&\sigma&\sigma\\
&&&&&&&&&\sigma&\sigma&\sigma\\
&&&&\times&\times&\cdots&\times&\times&\times\\
&&\times&\times&\times&\times&\cdots&\times&\times&\bullet\\
\tau&\tau&\tau
\end{matrix}
\longrightarrow
\begin{matrix}
&&&&&&&&&&&&\sigma\\
&&&&&&&&&\bullet&\sigma&\sigma&\sigma\\
&&&&&&&&&\sigma&\sigma&\sigma\\
&&&&\times&\times&\cdots&\times&\times&\sigma\\
&&\times&\times&\times&\times&\cdots&\times&\times&\times\\
\tau&\tau&\tau
\end{matrix}}
\end{equation*}
\caption{Map from $\tabset{\backwardA}$ to
                  ${\mathcal V}_3$ via a backward-slide}
\label{ShiftV3Eq}
\end{figure}
Let ${\mathcal V}_3$ be the set of all SSYTx that result in this
third instance.
As in the first case,
${\mathcal V}_1$ and ${\mathcal V}_2$ are defined to be the sets of
SSYTx of shapes $\backwardA^*$ and $\forwardA$ respectively that result from
the first two instances.
Then $\{\backwardA\}=g({\mathcal V}_1)+g({\mathcal V}_2)+g({\mathcal V}_3)$.
Moreover, $\{\backwardA^+\}=g({\mathcal V}_1)$ as in the first case,
resulting in
$\{\backwardA\}-\{\backwardA^+\}-g({\mathcal V}_3)=g({\mathcal V}_2)$.

The elements of ${\mathcal V}_3$ are now treated in analogy
with the treatment of ${\mathcal U}_3$ above.
Namely, we form the set ${\mathcal V}_3^\vee$ by,
for each tableau $V\in{\mathcal V}_3$, shifting downward
each entry in the columns of $V$ to the right of the vacancy.
This shift is depicted in Fig.~\ref{ShiftV3DEq}.
\begin{figure}[ht]
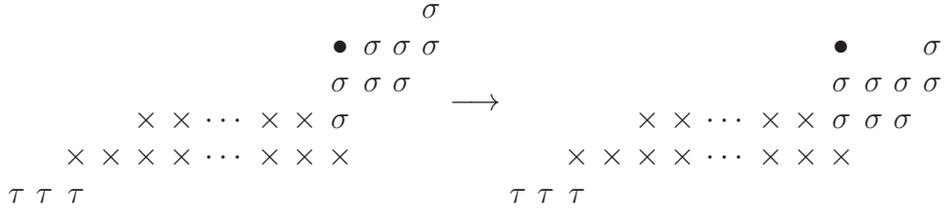

\begin{equation*}
{\setlength{\arraycolsep}{2pt}
\begin{matrix}
&&&&&&&&&&&&\sigma\\
&&&&&&&&&\bullet&\sigma&\sigma&\sigma\\
&&&&&&&&&\sigma&\sigma&\sigma\\
&&&&\times&\times&\cdots&\times&\times&\sigma\\
&&\times&\times&\times&\times&\cdots&\times&\times&\times\\
\tau&\tau&\tau
\end{matrix}
\longrightarrow
\begin{matrix}
\\
&&&&&&&&&\bullet&&&\sigma\\
&&&&&&&&&\sigma&\sigma&\sigma&\sigma\\
&&&&\times&\times&\cdots&\times&\times&\sigma&\sigma&\sigma\\
&&\times&\times&\times&\times&\cdots&\times&\times&\times\\
\tau&\tau&\tau
\end{matrix}}
\end{equation*}
\caption{Map from ${\mathcal V}_3$ to
                  ${\mathcal V}_3^\vee\subset\tabset{\backwardC}$
                  via down-shifting of some entries $\sigma$}
\label{ShiftV3DEq}
\end{figure}
We see that each element of ${\mathcal V}_3^\vee$ is a SSYT of shape
$$\backwardC=(\sigma\backslash\sigma_s,m+\sigma_s,n,\tau\,|\,\bar\sigma\backslash\bar\sigma_s,m+1,\bar\tau).$$
Indeed, by reversing the above construction, we obtain
${\mathcal V}_3^\vee=\tabset{\backwardC}$.
Therefore, $g({\mathcal V}_3^\vee)=g({\mathcal V}_3)=\{\backwardC\}$.
It now follows that
$\{\backwardA\}-\{\backwardA^+\}-\{\backwardC\}=g({\mathcal V}_2)$,
where we set $\{\backwardC\}=0$ if $\sigma$ is empty.

As in the previous cases, ${\mathcal U}_2$ and ${\mathcal V}_2$
are in bijection, whence $g({\mathcal U}_2)=g({\mathcal V}_2)$,
and thus
$\{\forwardA\}=\{\backwardA\}-\{\backwardA^+\}-\{\backwardC\}$.
This yields the required result in the case for which $x=m$ and $x<n-1$.


The case for which $x=m$ and $x=n-1$ is an amalgam of the two previous
cases.
Here, rows $(s+1)$ and $(s+2)$ of $T\in\tabset{\forwardA}$ are
flush at their right ends, but the vacancy $C$ has more than
one entry from $T$ below it.
Consequently, forward-sliding $C$ through $T$
necessarily results in this vacancy migrating to the end of
the $(s+2)$th row, as in Fig.~\ref{ShiftU2Eq},
or to the bottom of the column which originally contained $C$,
as in Fig.~\ref{ShiftU3Eq}.
Then, in comparison with the case for which $x<m$ and $x=n-1$,
the set ${\mathcal U}_1$ does not arise, and we obtain
$\{\forwardA\}-\{\forwardC\}=g({\mathcal U}_2)$.
Here again, if $\tau$ is empty, we set $\{\forwardC\}=0$.

Similarly, the backward-sliding process in this case is as in the case
for which $x=m$ and $x<n-1$, except that the set
${\mathcal V}_1$ does not arise, and consequently we obtain
$\{\backwardA\}-\{\backwardC\}=g({\mathcal V}_2)$.
Here again, if $\sigma$ is empty, we set $\{\backwardC\}=0$.

The familiar bijection between ${\mathcal U}_2$ and ${\mathcal V}_2$
then implies that
$\{\forwardA\}-\{\forwardC\}=\{\backwardA\}-\{\backwardC\}$,
giving the required result in this $x=m=n-1$ case,
thereby completing the proof of Lemma~\ref{BasLemSec}.
\end{proof}




In order to generalise this result and accommodate special cases of the
type appearing in \eqref{BasicEq}, it is convenient to introduce 
the skew Schur functions
$\Kfunc{m}{n}{x}$ that, for positive integers $m,n\geq1$, 
$0\le x\le\min\{m,n\}+1$ and
$\sigma,\tau,\ov\sigma,\ov\tau$ satisfying Hypothesis~\ref{ovsigma},
are defined as follows:
\begin{equation}
\label{Eq-K}
\Kfunc{m}{n}{x}=
\begin{cases}
\{\sigma,m,n,\tau\,|\,\bar\sigma,x,\bar\tau\}
  &\text{if $x\le m$ and $x\le n$;}\\
-\{\sigma\backslash\sigma_s,m+\sigma_s,n,\tau\,|\,
            \bar\sigma\backslash\bar\sigma_s,x,\bar\tau\}
  &\text{if $x=m+1$ and $x\le n$;}\\
-\{\sigma,m,n+\tau_1,\tau\backslash\tau_1\,|\,
            \bar\sigma,x,\bar\tau\backslash\bar\tau_1\}
  &\text{if $x\le m$ and $x=n+1$;}\\
\{\sigma\backslash\sigma_s,m+\sigma_s,n+\tau_1,\tau\backslash\tau_1\,|\,
            \bar\sigma\backslash\bar\sigma_s,x,\bar\tau\backslash\bar\tau_1\}
  &\text{if $x=m+1$ and $x=n+1$.}
\end{cases}
\end{equation}
If $s=0$ then the term containing
$\sigma\backslash\sigma_s$ is to be set to 0.
Similarly, if $t=0$ then the term
containing $\tau\backslash\tau_1$ is to be set to 0.

\begin{lemma}\label{GenLem}
Let $m,n,m',n'$ be positive integers for which $m+n=m'+n'$
and let $x,x'$ be non-negative integers for which
$x,x'\le\min\{m,n,m',n'\}+1$ then
\begin{equation}\label{BasCorEq}
\Kfunc{m}{n}{x}-\Kfunc{m'}{n'}{x}= \Kfunc{m}{n}{x'}-\Kfunc{m'}{n'}{x'}.
\end{equation}
\end{lemma}

\begin{proof}
We first write the equation \eqref{BasicEq} in the form
\begin{equation}\label{BasLemEq}
\Kfunc{p}{q}{y}-\Kfunc{p+1}{q-1}{y}= \Kfunc{p}{q}{y+1}-\Kfunc{p+1}{q-1}{y+1}
\end{equation}
where $p\ge1$, $q\ge2$, and $0\le y\le\min\{p,q-1\}$.
If, in addition, $y<y'\le\min\{p+1,q\}$, then repeated use of \eqref{BasLemEq}
yields
\begin{equation}\label{BasLemEq2}
\Kfunc{p}{q}{y}-\Kfunc{p+1}{q-1}{y}= \Kfunc{p}{q}{y'}-\Kfunc{p+1}{q-1}{y'}.
\end{equation}

Without loss of generality, we only need to prove \eqref{BasCorEq}
in the case for which $m<m'$ and $x<x'$. In this case
\begin{equation}\label{BasCorPf1}
\begin{split}
&\Kfunc{m}{n}{x}-\Kfunc{m'}{n'}{x}\\
&\qquad\qquad
=\sum_{i=m}^{m'-1}\left(\Kfunc{i}{m+n-i}{x}-\Kfunc{i+1}{m+n-i-1}{x}\right)\\
&\qquad\qquad
=\sum_{i=m}^{m'-1}\left(\Kfunc{i}{m+n-i}{x'}-\Kfunc{i+1}{m+n-i-1}{x'}\right)\\
&\qquad\qquad
=\Kfunc{m}{n}{x'}-\Kfunc{m'}{n'}{x'}\,
\end{split}
\end{equation}
where \eqref{BasLemEq2} has been used to convert each summand of the
first sum into the corresponding summand of the second sum.
This proves \eqref{BasCorEq}.
\end{proof}

\subsection{An algebraic approach}

As an alternative to the above combinatorial approach to Lemmas~\ref{BasLemSec} 
and~\ref{GenLem}, an algebraic approach may be based on the use of
the Jacobi-Trudi determinants, \eqref{skew-schur-def-h} or \eqref{GeneralJ-T}, which
express a skew Schur function in terms of complete symmetric functions.

First, it is convenient to introduce the following family of 
determinants:
\begin{equation}
\label{Eq-H}
\begin{array}{l}
\Hfunc{m}{n}{z}=\\
\\
\left|
\begin{array}{llllllll}
h_{S_{1,1}}&\cdots&h_{S_{1,s}}&h_{S_{1,s}+m}
&h_{S_{1,s}+m+n-z}&h_{S_{1,s}+m+n-z+T_{1,1}}&\cdots&h_{S_{1,s}+m+n-z+T_{1,t}}\\
\vdots&\ddots&\vdots&\vdots&\vdots&\vdots&\ddots&\vdots\\
h_{S_{s,1}}&\cdots&h_{S_{s,s}}&h_{S_{s,s}+m}
&h_{S_{s,s}+m+n-z}&h_{S_{s,s}+m+n-z+T_{1,1}}&\cdots&h_{S_{s,s}+m+n-z+T_{1,t}}\\
0&\cdots&1&h_{m}
&h_{m+n-z}&h_{m+n-z+T_{1,1}}&\cdots&h_{m+n-z+T_{1,t}}\\
0&\cdots&0&h_{z}
&h_{n}&h_{n+T_{1,1}}&\cdots&h_{n+T_{1,t}}\\
0&\cdots&0&0
&1&h_{T_{1,1}}&\cdots&h_{T_{1,t}}\\
\vdots&\ddots&\vdots&\vdots
&\vdots&\vdots&\ddots&\vdots\\
0&\cdots&0&0
&0&h_{T_{t,1}}&\cdots&h_{T_{t,t}}\\
\end{array}
\right|
\end{array}
\end{equation}
where $S$ and $T$ are $s\times s$ and $t\times t$ matrices of integers, with
$s$ and $t$ non-negative integers, while $m$ and $n$ are positive integers and $z$
is any integer. As usual $h_k=0$ for $k<0$ and $h_0=1$. It is to be understood
that any sequence $0 \cdots 1$, whether vertical or horizontal, is a sequence 
of $0\,$s followed by a single $1$ in the position indicated.

With this definition we have the following:

\begin{lemma}
\label{Lem-Hdiff}
Let $m,n,m',n'$ be positive integers such that $m+n=m'+n'$. Then the difference
\begin{equation}\label{Eq-Hdiff}
\Hfunc{m}{n}{z}-\Hfunc{m'}{n'}{z}
\end{equation}
is independent of $z$ for all integers $z$. 
\end{lemma}

\begin{proof} 
Since $\Hfunc{m}{n}{z}$ is the determinant of a matrix that is block triangular, 
save for the single subdiagonal element $h_z$, it may be evaluated by adding
to the determinant of the block triangular matrix the contribution arising from
$h_z$ and its cofactor. This follows from the fact that the expansion
of the determinant $\Hfunc{m}{n}{z}$ is linear in $h_z$, that is to say
of the form $X+h_z\,Y$, with the term $X$ calculated by setting $h_z=0$
and the term $Y$ equal to the cofactor of $h_z$. 
Since, in addition, the determinant of a block triangular
matrix is the product of the determinant of its diagonal blocks, we have
\begin{equation}\label{Eq-H-expansion}
\begin{array}{l}
\Hfunc{m}{n}{z}=\\
\\
\left|
\begin{array}{llll}
h_{S_{1,1}}&\cdots&h_{S_{1,s}}&h_{S_{1,s}+m}\\
\vdots&\ddots&\vdots&\vdots\\
h_{S_{s,1}}&\cdots&h_{S_{s,s}}&h_{S_{s,s}+m}\\
0&\cdots&1&h_{m}\\
\end{array}
\right|
\ \cdot\ 
\left|
\begin{array}{llll}
h_{n}&h_{n+T_{1,1}}&\cdots&h_{n+T_{1,t}}\\
1&h_{T_{1,1}}&\cdots&h_{T_{1,t}}\\
\vdots&\vdots&\ddots&\vdots\\
0&h_{T_{t,1}}&\cdots&h_{T_{t,t}}\\
\end{array}
\right|
\\
\\
\ - \ h_{z}\ \cdot\ 
\left|
\begin{array}{lllllll}
h_{S_{1,1}}&\cdots&h_{S_{1,s}}
&h_{S_{1,s}+m+n-z}&h_{S_{1,s}+m+n-z+T_{1,1}}&\cdots&h_{S_{1,s}+m+n-z+T_{1,t}}\\
\vdots&\ddots&\vdots&\vdots&\vdots&\ddots&\vdots\\
h_{S_{s,1}}&\cdots&h_{S_{s,s}}
&h_{S_{s,s}+m+n-z}&h_{S_{s,s}+m+n-z+T_{1,1}}&\cdots&h_{S_{s,s}+m+n-z+T_{1,t}}\\
0&\cdots&1
&h_{m+n-z}&h_{m+n-z+T_{1,1}}&\cdots&h_{m+n-z+T_{1,t}}\\
0&\cdots&0
&1&h_{T_{1,1}}&\cdots&h_{T_{1,t}}\\
\vdots&\ddots&\vdots
&\vdots&\vdots&\ddots&\vdots\\
0&\cdots&0
&0&h_{T_{t,1}}&\cdots&h_{T_{t,t}}\\
\end{array}
\right|\,.
\end{array}
\end{equation}

It can be seen from this that for all $m,n,m',n'\in\N$ with
$m+n=m'+n'$ and for all integers $z$ we have the identity
\begin{equation}
\label{Eq-H-identity}
\begin{array}{l}
\Hfunc{m}{n}{z}-\Hfunc{m'}{n'}{z}\\
\\
=
\left|
\begin{array}{llll}
h_{S_{1,1}}&\cdots&h_{S_{1,s}}&h_{S_{1,s}+m}\\
\vdots&\ddots&\vdots&\vdots\\
h_{S_{s,1}}&\cdots&h_{S_{s,s}}&h_{S_{s,s}+m}\\
0&\cdots&1&h_{m}\\
\end{array}
\right|
\ \cdot\ 
\left|
\begin{array}{llll}
h_{n}&h_{n+T_{1,1}}&\cdots&h_{n+T_{1,t}}\\
1&h_{T_{1,1}}&\cdots&h_{T_{1,t}}\\
\vdots&\vdots&\ddots&\vdots\\
0&h_{T_{t,1}}&\cdots&h_{T_{t,t}}\\
\end{array}
\right|
\\
\\
-
\left|
\begin{array}{llll}
h_{S_{1,1}}&\cdots&h_{S_{1,s}}&h_{S_{1,s}+m'}\\
\vdots&\ddots&\vdots&\vdots\\
h_{S_{s,1}}&\cdots&h_{S_{s,s}}&h_{S_{s,s}+m'}\\
0&\cdots&1&h_{m'}\\
\end{array}
\right|
\ \cdot\ 
\left|
\begin{array}{llll}
h_{n'}&h_{n'+T_{1,1}}&\cdots&h_{n'+T_{1,t}}\\
1&h_{T_{1,1}}&\cdots&h_{T_{1,t}}\\
\vdots&\vdots&\ddots&\vdots\\
0&h_{T_{t,1}}&\cdots&h_{T_{t,t}}\\
\end{array}
\right|\,
\end{array}
\end{equation}
where it is to be noted that the resulting expression is independent of $z$,
as claimed.
\end{proof}

To make the connection with Lemmas \ref{BasLemSec} and \ref{GenLem}
by means of Jacobi-Trudi determinants, we introduce the following
hypothesis.

\begin{hypothesis}
\label{Hyp-ST}
Let $\sigma,\tau,\ov\sigma,\ov\tau$ satisfy Hypothesis~\ref{ovsigma},
and let the matrices $S$ and $T$ have matrix elements 
defined in terms of these by
\begin{equation}\label{Eq-S-sigma}
   S_{ij}=\left\{
\begin{array}{ll}
 (\sigma_i+\sigma_{i+1}+\cdots+\sigma_{j})-(\ov\sigma_i+\ov\sigma_{i+1}+\cdots+\ov\sigma_{j-1})-i+j &\text{if $i\leq j-1$;}\\
 \  \sigma_i &\text{if $i=j$;}\\
 \  \ov\sigma_j-1  &\text{if $i=j+1$;}\\
 -(\sigma_{j+1}+\sigma_{j+2}+\cdots+\sigma_{i-1})+(\ov\sigma_j+\ov\sigma_{j+1}+\cdots+\ov\sigma_{i-1})-i+j&\text{if $i> j+1$,}\\
 \end{array} \right. \,
\end{equation}
for $1\leq i,j\leq s$, and
\begin{equation}\label{Eq-T-tau}
   T_{ij}=\left\{
\begin{array}{ll}
 (\tau_i+\tau_{i+1}+\cdots+\tau_{j})-(\ov\tau_i+\ov\tau_{i+1}+\cdots+\ov\tau_{j-1})-i+j &\text{if $i\leq j-1$;}\\
 \  \tau_i &\text{if $i=j$;}\\
 \  \ov\tau_j-1  &\text{if $i=j+1$;}\\
 -(\tau_{j+1}+\tau_{j+2}+\cdots+\tau_{i-1})+(\ov\tau_j+\ov\tau_{j+1}+\cdots+\ov\tau_{i-1})-i+j&\text{if $i> j+1$,}\\
 \end{array} \right. \,
\end{equation}
for $1\leq i,j\leq t$.
\end{hypothesis}

We then have
\begin{lemma}\label{HKlemma} 
Assume that $S$ and $T$ satisfy Hypothesis~\ref{Hyp-ST}.
If $0\le x\le\min\{m,n\}+1$ then
\begin{equation}
\label{Eq-HK}
\Hfunc{m}{n}{x-1}=\Kfunc{m}{n}{x}.
\end{equation}
\end{lemma}

\begin{proof}
Let $z=x-1$. 
In the case $z<\min\{m,n\}$, under the given hypothesis regarding $S$ and $T$, 
comparing the determinant \eqref{Eq-H} 
with the general Jacobi-Trudi determinant \eqref{GeneralJ-T} for a skew Schur function, 
expressed in overlap notation by way of \eqref{GeneralMatrix}, 
and noting that $\sigma_s=\tau_1=1$, 
immediately gives
\begin{equation}
\Hfunc{m}{n}{x-1}=
\{\sigma,m,n,\tau\,|\,\bar\sigma,x,\bar\tau\}.
\end{equation}
Thanks to the definition \eqref{Eq-K}, this proves \eqref{Eq-HK} in the case
$x\leq\min\{m,n\}$.

In the $z=m<n$ case, if $s=0$ then $\Hfunc{m}{n}{m}=0$ since 
the first and second rows of determinant \eqref{Eq-H}  coincide. 
On the other hand, for $s>0$ the entries of $\Hfunc{m}{n}{m}$ in the 
$(s+1)$th and $(s+2)$th rows all coincide apart from a single entry $1$ 
in the $(s+1)$th row.
Consideration of the cofactor of this element then yields
\begin{equation}
\label{Eq-STm}
\Hfunc{m}{n}{m}=
-\{\sigma\backslash\sigma_s,m+\sigma_s,n,\tau\,|\,\ov\sigma\backslash\ov\sigma_s,m+1,\ov\tau\}
\end{equation}
for $s>0$.
Once again, the definition \eqref{Eq-K} is such that this proves \eqref{Eq-HK} in the case
$x-1=m<n$.
 
Similarly in the $z=n<m$ case, if $t=0$ then $\Hfunc{m}{n}{n}=0$
since the $(s+1)$th and $(s+2)$th columns coincide. 
For $t>0$ the entries of $\Hfunc{m}{n}{n}$ in the $(s+1)$th and $(s+2)$th 
columns coincide apart from a single entry $1$ in the $(s+2)$th column. This time,
consideration of the cofactor of this element yields
\begin{equation}
\label{Eq-STn}
\Hfunc{m}{n}{n}
=
-\{\sigma,m,n+\tau_1,\tau\backslash\tau_1\,|\,\ov\sigma,n+1,\ov\tau\backslash \ov\tau_1\}
\end{equation}
for $t>0$.
Via the definition \eqref{Eq-K}, this proves \eqref{Eq-HK} in the case
$x-1=n<m$.

Finally, the case $z=m=n$ is such that $\Hfunc{m}{m}{m}=0$ if either
$s=0$ or $t=0$ since the determinant involves two identical
rows or columns, respectively. 
For $s,t>0$ the determinant involves two rows identical save for a single entry $1$
and two columns identical save for another single entry $1$. Subtracting rows,
taking the cofactor of the $1$, and then doing the same for the columns
gives
\begin{equation}
\Hfunc{m}{m}{m}=\{\sigma\backslash\sigma_s,m+\sigma_s,m+\tau_1,\tau\backslash\tau_1
\,|\, \ov\sigma\backslash\ov\sigma_s,m+1,\ov\tau\backslash\ov\tau_1\} \,
\end{equation}
for $s>0$ and $t>0$.
Again via \eqref{Eq-K}, this proves \eqref{Eq-HK} in the final $x-1=n=m$ case.
\end{proof}

This identification of $\Hfunc{m}{n}{x-1}$, for $S$ and $T$ satisfying 
Hypothesis~\ref{Hyp-ST}, offers us an alternative proof of Lemma~\ref{BasLemSec}.

\begin{proof}[Alternative proof of Lemma~\ref{BasLemSec}]
For all $s\times s$ and $t\times t$ matrices $S$ and $T$, 
it follows from Lemma~\ref{Lem-Hdiff} that for all $m\geq1$,
$n\geq2$ and all integers $x$ we have
\begin{equation}
\Hfunc{m}{n}{x-1}-\Hfunc{m+1}{n-1}{x-1}=\Hfunc{m}{n}{x}-\Hfunc{m+1}{n-1}{x}\,.
\nonumber
\end{equation}
If we then assume that $S$ and $T$ satisfy Hypothesis~\ref{Hyp-ST},
and restrict $x$ so that $0\leq x\leq \min\{m,n-1\}$, we can
apply Lemma~\ref{HKlemma} to each term. This gives
\begin{equation}
\Kfunc{m}{n}{x}-\Kfunc{m+1}{n-1}{x}=\Kfunc{m}{n}{x+1}-\Kfunc{m+1}{n-1}{x+1}\,.
\end{equation}
However, thanks to the definition \eqref{Eq-K}, this is just \eqref{BasicEq}.
\end{proof}

In fact, it is no harder to prove Lemma~\ref{GenLem}.

\begin{proof}[Alternative proof of Lemma~\ref{GenLem}]
Once again for all $s\times s$ and $t\times t$ matrices $S$ and $T$, 
it follows from Lemma~\ref{Lem-Hdiff} that for all positive integers
$m,n,m',n'$ such that $m+n=m'+n'$, and all integers $x,x'$ we have
\begin{equation}
\Hfunc{m}{n}{x-1}-\Hfunc{m'}{n'}{x-1}=\Hfunc{m}{n}{x'-1}-\Hfunc{m'}{n'}{x'-1}\,.
\nonumber
\end{equation}
Then assuming that $S$ and $T$ satisfy Hypothesis~\ref{Hyp-ST},
and restricting $x$ and $x'$ so that $0\leq x,x'\leq\min\{m,n,m',n'\}+1$, we can
simply apply Lemma~\ref{HKlemma} to each term, giving
\begin{equation}
\Kfunc{m}{n}{x}-\Kfunc{m'}{n'}{x}=\Kfunc{m}{n}{x'}-\Kfunc{m'}{n'}{x'}\,
\end{equation}
as required.
\end{proof}

\subsection{Simple skew Schur function differences}  

Now we derive four straightforward corollaries of Lemma \ref{GenLem}.  

\begin{corollary}\label{Cor1}
Assume $\sigma ,\tau ,\ov\sigma ,\ov\tau$ satisfy Hypothesis~\ref{ovsigma}. If $0\le x\le m$ then
\begin{equation}\label{Cor1Eq}
\begin{split}
&\{\sigma,m,m+1,\tau\,|\,\sbar,x,\tbar\}
-\{\sigma,m+1,m,\tau\,|\,\sbar,x,\tbar\}\\
&\qquad\qquad\quad=\quad
{}-\{\sigma\backslash\sigma_s,m+\sigma_s,m+1,\tau\,|\,\sslash,m+1,\tbar\}\\
&\qquad\qquad\qquad\qquad
+\{\sigma,m+1,m+\tau_1,\tau\backslash\tau_1\,|\,\sbar,m+1,\tslash\}.
\end{split}
\end{equation}
If  $s=0$ then omit the first term.
Similarly, if  $t=0$ then omit the second term.
\end{corollary}

\begin{proof}
This is the special case
$$\Kfunc{m}{m+1}{x}-\Kfunc{m+1}{m}{x}=\Kfunc{m}{m+1}{m+1}-\Kfunc{m+1}{m}{m+1}$$
of Lemma~\ref{GenLem},
rewritten using \eqref{Eq-K}.
Note that the terms on the left are each the first case of \eqref{Eq-K},
and the terms on the right are the second and
third cases, respectively.
\end{proof}

The next corollary was proved independently by McNamara \cite{McN}.

\begin{corollary}\label{Cor2} 
Assume $\sigma ,\tau ,\ov\sigma ,\ov\tau$ satisfy Hypothesis~\ref{ovsigma}. If $0\le x<n\le m$ then
\begin{equation}\label{Cor2Eq}
\begin{split}
&\{\sigma,m,n,\tau\,|\,\sbar,x,\tbar\}
-\{\sigma,m+1,n-1,\tau\,|\,\sbar,x,\tbar\}\\
&\qquad\qquad\quad=\quad
\{\sigma,m,n,\tau\,|\,\sbar,n,\tbar\}\\
&\qquad\qquad\qquad\qquad
+\{\sigma,m+1,n-1+\tau_1,\tau\backslash\tau_1\,|\,\sbar,n,\tslash\}.
\end{split}
\end{equation}
If  $t=0$ then omit the second term.
\end{corollary}
\begin{proof}
This is the special case
$$\Kfunc{m}{n}{x}-\Kfunc{m+1}{n-1}{x}=\Kfunc{m}{n}{n}-\Kfunc{m+1}{n-1}{n}$$
of Lemma~\ref{GenLem},
rewritten using \eqref{Eq-K}.
Note that the terms on the left are each the first case of \eqref{Eq-K},
and the terms on the right are the first and third cases, respectively.
\end{proof}

\begin{corollary}\label{Cor4}
Assume $\sigma ,\tau ,\ov\sigma ,\ov\tau$ satisfy Hypothesis~\ref{ovsigma}. If $0<m\le n$ then
\begin{equation}\label{Cor4Eq}
\begin{split}
&\{\sigma,m,n,\tau\,|\,\sbar,m-1,\tbar\}
-\{\sigma,n+1,m-1,\tau\,|\,\sbar,m-1,\tbar\}\\
&\qquad\qquad\quad=\quad
\{\sigma,m,n,\tau\,|\,\sbar,m,\tbar\}\\
&\qquad\qquad\qquad\qquad
+\{\sigma,n+1,m-1+\tau_1,\tau\backslash\tau_1\,|\,\sbar,m,\tslash\}.
\end{split}
\end{equation}
If  $t=0$ then omit the second term.
\end{corollary}

\begin{proof}
This is the special case
$$\Kfunc{m}{n}{m-1}-\Kfunc{n+1}{m-1}{m-1}=\Kfunc{m}{n}{m}-\Kfunc{n+1}{m-1}{m}$$
of Lemma~\ref{GenLem},
rewritten using \eqref{Eq-K}.
Note that the terms on the left are each the first case of \eqref{Eq-K},
and the terms on the right are the first and third cases, respectively.
\end{proof}

\begin{corollary}\label{Cor3}
Assume $\sigma ,\tau ,\ov\sigma ,\ov\tau$ satisfy Hypothesis~\ref{ovsigma}. If $0\le m<n$ then
\begin{equation}\label{Cor3Eq}
\begin{split}
&\{\sigma,m,n,\tau\,|\,\sbar,m,\tbar\}
-\{\sigma,n,m,\tau\,|\,\sbar,m,\tbar\}\\
&\qquad\qquad\quad=\quad
{}-\{\sigma\backslash\sigma_s,m+\sigma_s,n,\tau\,|\,\sslash,m+1,\tbar\}\\
&\qquad\qquad\qquad\qquad
+\{\sigma,n,m+\tau_1,\tau\backslash\tau_1\,|\,\sbar,m+1,\tslash\}.
\end{split}
\end{equation}
If  $s=0$ then omit the first term.
Similarly, if  $t=0$ then omit the second term.
\end{corollary}
\begin{proof}
This is the special case
$$\Kfunc{m}{n}{m}-\Kfunc{n}{m}{m}=\Kfunc{m}{n}{m+1}-\Kfunc{n}{m}{m+1}$$
of Lemma~\ref{GenLem},
rewritten using \eqref{Eq-K}.
Note that the terms on the left are each the first case of \eqref{Eq-K},
and the terms on the right are the second and third cases, respectively.
\end{proof}

\subsection{Fundamental Schur positive difference expression}
We are now close to proving our main theorem, Theorem~\ref{big-difference}. The following two lemmas, which are consequences of the above
four corollaries, facilitate the proof of the equation therein, \eqref{DifferenceEq}.

\begin{lemma}\label{Lem1}
Assume $\sigma ,\tau ,\ov\sigma ,\ov\tau$ satisfy Hypothesis~\ref{ovsigma}. If $M\ge1$, $a\ge2$ and $k\ge1$ then
\begin{equation}
\nonumber
\begin{split}
&  \{\sigma,ka,ka,a^{M},\tau\,|\,
           \sbar,(k-1)a+1,1^{M}, \tbar\}\\[2mm]
& - \{\sigma,ka+1,ka-1,a^{M},\tau\,|\,
           \sbar,(k-1)a+1,1^{M}, \tbar\}\\
& + \sum_{i=1}^{M}
  \left( \{\sigma,a+1,a^{i-1},ka-1,ka,a^{M-i},\tau\,|\,
           \sbar, 1^{i},(k-1)a+1,1^{M-i}, \tbar\} \right.\\
&\quad\qquad
  - \left. \{\sigma,a+1,a^{i-1},ka,ka-1,a^{M-i},\tau\,|\,
           \sbar, 1^{i},(k-1)a+1,1^{M-i}, \tbar\} \right)\\[2mm]
&= \{\sigma,ka,ka,a^{M},\tau\,|\,
           \sbar,ka,1^{M}, \tbar\}\\[2mm]
&\quad+ \{\sigma,ka+1,(k+1)a-1,a^{M-1},\tau\,|\,
           \sbar,ka,1^{M-1}, \tbar\}\\[2mm]
&\quad- \{\sigma,(k+1)a,ka,a^{M-1},\tau\,|\,
           \sbar,ka,1^{M-1}, \tbar\}\\[2mm]
&\quad+ \{\sigma,a+1,a^{M-1},ka,ka-1+\tau_1,\tau\backslash\tau_1\,|\,
           \sbar,1^{M},ka,\tslash\}\\
&\quad+
\sum_{i=1}^{M-1}
  \left( \{\sigma,a+1,a^{i-1},ka,(k+1)a-1,a^{M-1-i},\tau\,|\,
           \sbar, 1^{i},ka,1^{M-i-1}, \tbar\} \right.\\
&\qquad\qquad
  - \left. \{\sigma,a+1,a^{i-1},(k+1)a-1,ka,a^{M-1-i},\tau\,|\,
           \sbar, 1^{i},ka,1^{M-i-1}, \tbar\} \right)
\end{split}
\end{equation}
where the fourth term on the right is to be omitted
if  $t=0$.
\end{lemma}

\begin{proof}
The first two terms on the right result from applying
Corollary \ref{Cor2} to the first two terms on the left,
using $m=n=ka$ and $x=(k-1)a+1$.
The other terms result from applying Corollary \ref{Cor1}
to each summand on the left side, in turn, using $m=ka-1$
and $x=(k-1)a+1$.
The positive term in the $i$th summand on the right
is the positive term of \eqref{Cor1Eq} applied to the $i$th
summand on the left, and the negative term in the $i$th
summand on the right is the negative term of \eqref{Cor1Eq}
applied to the $(i+1)$th summand on the left.
The third term on the right is the negative term of
\eqref{Cor1Eq} applied to the first summand on the left,
and the fourth term on the right is the positive term of
\eqref{Cor1Eq} applied to the $M$th summand on the left.
\end{proof}

\begin{lemma}\label{Lem2}
Assume $\sigma ,\tau ,\ov\sigma ,\ov\tau$ satisfy Hypothesis~\ref{ovsigma}. If $M\ge1$, $a\ge2$ and $k\ge1$ then
\begin{equation}
\nonumber
\begin{split}
&  \{\sigma,ka+1,(k+1)a-1,a^{M},\tau\,|\,
           \sbar,ka,1^{M},\tbar\}\\[2mm]
& - \{\sigma,(k+1)a,ka,a^{M},\tau\,|\,
           \sbar,ka,1^{M}, \tbar\}\\
& + \sum_{i=1}^{M}
  \left( \{\sigma,a+1,a^{i-1},ka,(k+1)a-1,a^{M-i},\tau\,|\,
           \sbar,1^{i},ka,1^{M-i},\tbar\} \right.\\
&\quad\qquad
  - \left. \{\sigma,a+1,a^{i-1},(k+1)a-1,ka,a^{M-i},\tau\,|\,
           \sbar, 1^{i},ka,1^{M-i},\tbar\} \right)\\[2mm]
&= \{\sigma,ka+1,(k+1)a-1,a^{M},\tau\,|\,
           \sbar,ka+1,1^{M}, \tbar\}\\[2mm]
&\quad+ \{\sigma,(k+1)a,(k+1)a,a^{M-1},\tau\,|\,
           \sbar,ka+1,1^{M-1},\tbar\}\\[2mm]
&\quad- \{\sigma,(k+1)a+1,(k+1)a-1,a^{M-1},\tau\,|\,
           \sbar,ka+1,1^{M-1}, \tbar\}\\[2mm]
&\quad+ \{\sigma,a+1,a^{M-1},(k+1)a-1,ka+\tau_1,\tau\backslash\tau_1\,|\,
           \sbar, 1^{M},ka+1,\tslash\}\\
&\quad+
\sum_{i=1}^{M-1}
  \left( \{\sigma,a+1,a^{i-1},(k+1)a-1,(k+1)a,a^{M-1-i},\tau\,|\,
           \sbar, 1^{i},ka+1,1^{M-i-1}, \tbar\} \right.\\
&\qquad\qquad
  - \left. \{\sigma,a+1,a^{i-1},(k+1)a,(k+1)a-1,a^{M-1-i},\tau\,|\,
           \sbar, 1^{i},ka+1,1^{M-i-1}, \tbar\} \right)
\end{split}
\end{equation}
where the fourth term on the right is to be omitted
if  $t=0$.
\end{lemma}

\begin{proof}
The first two terms on the right result from applying
Corollary \ref{Cor4} to the first two terms on the left,
using $m=ka+1$ and $n=(k+1)a-1$.
The other terms result from applying Corollary \ref{Cor3}
to each summand on the left side, in turn, using $m=ka$
and $n=(k+1)a-1$.
The positive term in the $i$th summand on the right
is the positive term of \eqref{Cor3Eq} applied to the $i$th
summand on the left, and the negative term in the $i$th
summand on the right is the negative term of \eqref{Cor3Eq}
applied to the $(i+1)$th summand on the left.
The third term on the right is the negative term of
\eqref{Cor3Eq} applied to the first summand on the left,
and the fourth term on the right is the positive term of
\eqref{Cor3Eq} applied to the $M$th summand on the left.
\end{proof}

We now come to our main theorem.

\begin{theorem}\label{big-difference}
Assume $\sigma ,\tau ,\ov\sigma ,\ov\tau$ satisfy Hypothesis~\ref{ovsigma}.
Then, for $a\ge2$ and $n\ge2$
\begin{equation}\label{DifferenceEq}
\begin{split}
&\!\!
\{\sigma , a^n, \tau \, | \, \ov\sigma,1^{n-1} , \ov\tau\}\
- \ \{\sigma , a+1, a^{n-2}, a-1, \tau \, |\,  \ov\sigma,1^{n-1} , \ov\tau\}\\
& =\quad\sum_{k=1}^{\left\lfloor\frac{n}{2}\right\rfloor}
  \{\sigma,ka,ka,a^{n-2k},\tau\,|\,\sbar,ka,1^{n-2k},\tbar\}\\
& \quad+\sum_{k=1}^{\left\lfloor\frac{n-1}{2}\right\rfloor}
  \{\sigma,a+1,a^{n-2k-1},ka,ka-1+\tau_1,\tau\backslash\tau_1\,|\,
             \sbar, 1^{n-2k},ka,\tslash\}\\
& \quad+\sum_{k=1}^{\left\lfloor\frac{n-1}{2}\right\rfloor}
  \{\sigma,ka+1,(k+1)a-1,a^{n-2k-1},\tau\,|\,\sbar,ka+1,1^{n-2k-1}, \tbar\}\\
& \quad+\sum_{k=1}^{\left\lfloor\frac{n-2}{2}\right\rfloor}
  \{\sigma,a+1,a^{n-2k-2},(k+1)a-1,ka+\tau_1,\tau\backslash\tau_1\,|\,
             \sbar, 1^{n-2k-1},ka+1,\tslash\}\\[1mm]
& \quad+
\begin{cases}
  \{\sigma,\frac{n+1}2a,\frac{n-1}2a+\tau_1,\tau\backslash\tau_1\,|\,
             \sbar,\frac{n-1}2a+1,\tslash\}
                      &\text{if $n$ is odd;}\\[1.5mm]
  \{\sigma,\frac n2a+1,\frac n2a-1+\tau_1,\tau\backslash\tau_1\,|\,
             \sbar,\frac n2a,\tslash\}
                      &\text{if $n$ is even.}
\end{cases}
\end{split}
\end{equation}
If  $s=0$, then omit the $\sigma$ from
each expression.
If  $t=0$, then omit the terms in the
2nd, 4th, 5th and 6th lines,
and in the remaining terms, omit the $\tau$.
If $t=1$, then omit the $\tau\backslash\tau_1$ from the terms in
the 2nd, 4th, 5th and 6th lines.
The right side of \eqref{DifferenceEq} then has $2n-2$ terms
when $t>0$, and $n-1$ terms when $t=0$.\end{theorem}

\begin{example}\label{useful-ex} Before we prove this relation we
provide
four examples involving ribbon Schur functions. Note how the results
differ
for the cases $s=0=t$, $s\neq 0=t$, $s=0\neq t$ and $s\neq 0\neq t$.
$$r_{(2,2,2)}-r_{(3,2,1)}=\{2,2,2\, |\, 2,1\}+\{3,3\, |\, 3\}.$$
$$r_{(3,2,2,2)}-r_{(3,3,2,1)}=\{3,2,2,2\, |\, 1,2,1\}+\{3,3,3\, |\,
1,3\}.$$
$$r_{(2,2,2,1)}-r_{(3,2,1,1)}=\{2,2,2,1\ |\ 2,1,1\}+\{3,2,2\ |\
1,2\}+\{3,3,1\ |\ 3,1\}+\{4,3\ |\ 3\}.$$
$$r_{(3,2,2,2,1)}-r_{(3,3,2,1,1)}=\{3,2,2,2,1\ |\
1,2,1,1\}+\{3,3,2,2\ |\
1,1,2\}+\{3,3,3,1\ |\ 1,3,1\}+\{3,4,3\ |\ 1,3\}.$$
\end{example}


%

\begin{proof}[Proof of Theorem \ref{big-difference}]
To obtain expression \eqref{DifferenceEq} we first write
\begin{equation}\label{SplitEq}
\begin{split}
&\{\sigma , a^n, \tau \, | \, \ov\sigma,1^{n-1} , \ov\tau\}\ - \ \{\sigma , a+1, a^{n-2}, a-1, \tau \, |\,  \ov\sigma,1^{n-1} , \ov\tau\}\\[2mm]
&\qquad=
\{\sigma,a^n,\tau\,|\,\sbar, 1^{n-1}, \tbar\}
- \{\sigma,a+1,a-1,a^{n-2},\tau\,|\,\sbar, 1^{n-1}, \tbar\}\\
&\qquad\qquad+\sum_{i=1}^{n-2}
 \left(
   \{\sigma,a+1,a^{i-1},a-1,a,a^{n-2-i},\tau\,|\,\sbar, 1^{n-1}, \tbar\}\right.\\
&\qquad\qquad\qquad-
 \left.
   \{\sigma,a+1,a^{i-1},a,a-1,a^{n-2-i},\tau\,|\,\sbar, 1^{n-1}, \tbar\}\right).
\end{split}
\end{equation}
Now, starting with Lemma~\ref{Lem1} in the case $M=n-2$ and $k=1$,
we
repeatedly apply Lemmas~\ref{Lem1} and \ref{Lem2} alternately.
Each application of Lemma~\ref{Lem1} decreases $M$ by $1$ but leaves
$k$ fixed,
and each application of Lemma~\ref{Lem2} decreases $M$ by $1$ and
increases
$k$ by $1$.
The final application is made with $M=1$,
and is of  Lemma~\ref{Lem1} with $k=\frac{n-1}2$ if $n$ is odd,
and of Lemma~\ref{Lem2} with $k=\frac{n-2}2$ if $n$ is even.

Following each application, the first and fourth terms on the
right side of the expression in Lemma \ref{Lem1} or Lemma \ref{Lem2}
contribute to the final expression \eqref{DifferenceEq}.
The remaining terms (second, third and the summation)
are then acted upon by the subsequent application of
Lemma \ref{Lem1} or Lemma \ref{Lem2}.
When applying Lemma \ref{Lem1}, the first and
fourth terms on the right give rise to the $k$th summands
in the first and second terms on the right of \eqref{DifferenceEq}.
The summands $1\le k\le \lfloor\frac{n-1}2\rfloor$ of these
latter two terms arise in this way.
Similarly, when applying Lemma \ref{Lem2}, the first and
fourth terms on the right give rise to the $k$th summands
in the third and fourth terms on the right of \eqref{DifferenceEq}.
The summands $1\le k\le \lfloor\frac{n-2}2\rfloor$ of these
latter two terms arise in this way.

If $n$ is odd, the final application of Lemma \ref{Lem1} is
made with $M=1$ and $k=\frac{n-1}2$.
It remains to deal with the second and third terms that arise
on the right of the expression in Lemma \ref{Lem1}.
This is accomplished using Corollary \ref{Cor4}, which yields
\begin{equation}
\nonumber
\begin{split}
&\{\sigma,ka+1,(k+1)a-1,\tau\,|\,
           \sbar,ka,\tbar\}
- \{\sigma,(k+1)a,ka,\tau\,|\,
           \sbar,ka,\tbar\}\\[1mm]
&\qquad=
\{\sigma,ka+1,(k+1)a-1,\tau\,|\,
           \sbar,ka+1,\tbar\}\\
&\qquad\qquad+
 \{\sigma,(k+1)a,ka+\tau_1,\tau\backslash\tau_1\,|\,
           \sbar,ka+1,\tslash\}.
\end{split}
\end{equation}
The first term on the right here gives the additional
$k=\frac {n-1}2$ summand of
the third term on the right of \eqref{DifferenceEq},
and the second term gives the odd $n$ case of the final term.

If $n$ is even, the final application of Lemma \ref{Lem2} is
made with $M=1$ and $k=\frac{n-2}2$.
It remains to deal with the second and third terms that arise
on the right of the expression in Lemma \ref{Lem2}.
This is accomplished using Corollary \ref{Cor2}, which yields
\begin{equation}
\nonumber
\begin{split}
&\{\sigma,(k+1)a,(k+1)a,\tau\,|\,
           \sbar,ka+1,\tbar\}
- \{\sigma,(k+1)a+1,(k+1)a-1,\tau\,|\,
           \sbar,ka+1,\tbar\}\\[1mm]
&\qquad=
\{\sigma,(k+1)a,(k+1)a,\tau\,|\,
           \sbar,(k+1)a,\tbar\}\\
&\qquad\qquad+
 \{\sigma,(k+1)a+1,(k+1)a-1+\tau_1,\tau\backslash\tau_1\,|\,
           \sbar,(k+1)a,\tslash\}.
\end{split}
\end{equation}
The first term on the right here gives the additional $k=\frac n2$
summand of the first term on the right of \eqref{DifferenceEq},
and the second term gives the even $n$ case of the final term.
This completes the proof of \eqref{DifferenceEq}.\end{proof}

\section{Schur positivity of ribbon and skew Schur functions}\label{sec:positivity}

We now apply our results from the previous section to derive some new differences of skew Schur functions that are Schur positive.

The crucial feature of the formulae \eqref{Cor2Eq} and \eqref{DifferenceEq} 
for the differences of certain skew Schur functions is that they are expressed as wholly positive sums of skew Schur functions, each of which is itself necessarily Schur positive. This leads immediately to the following theorem.
\begin{theorem}\label{big-posdifference}

Assume $\sigma ,\tau ,\ov\sigma ,\ov\tau$ satisfy Hypothesis~\ref{ovsigma}. If $a\geq b\geq 2$ then
$$\{\sigma ,a,b, \tau\, |\, \ov\sigma,1 , \ov\tau\}\ - \ \{\sigma , a+1, b-1, \tau\, | \, \ov\sigma,1 , \ov\tau\}$$and
$$\{\sigma , a^n, \tau \, | \, \ov\sigma,1^{n-1} , \ov\tau\}\ - \ \{\sigma , a+1, a^{n-2}, a-1, \tau \, |\,  \ov\sigma,1^{n-1} , \ov\tau\}$$are Schur positive, where $\{ \alpha \, |\, \beta\}$ denotes the skew Schur function $s_{\lambda /\mu}$ satisfying 
$r^{(1)}(\lambda /\mu)=\alpha$ and $r^{(2)}(\lambda /\mu)=\beta$.
\end{theorem}

The \emph{descent set} of a SYT, $T$, is the set of all
entries $i$ such that $i+1$ appears in a lower row than $i$.
With this in mind we recall the following relationship
between ribbon Schur functions $r_\alpha$ for $\alpha\vDash N$,
and Schur functions  $s_\lambda$ for $\lambda\vdash N$.
 
\begin{lemma}\cite[Theorem 7]{Gessel}\label{theorem-ribbon-schur}
Let $\alpha \vDash N$ then
\begin{equation}\label{ribbon-schur}
r_\alpha =\sum_{\lambda\vdash N} d_{\lambda\alpha} s_\lambda
\end{equation}
where $d_{\lambda\alpha}$ is the number of SYTx of shape $\lambda$
and descent set $S(\alpha)$.
 \end{lemma}
 
This lemma, together with Theorem~\ref{big-posdifference},  allows us to prove the following result first conjectured by McNamara \cite{McN},
which is analogous to Theorem~\ref{h-schur-positive}.

\begin{theorem}\label{positive-ribbon-theorem}
Let $\lambda , \mu\vdash N$ then
$$r_\mu -r _\lambda$$is Schur positive if and only if  $\mu \dom \lambda$ and $\ell(\lambda)=\ell(\mu)$.
\end{theorem}

\begin{proof} First we prove that if   $\mu \not\dom \lambda$ or $\ell(\lambda)\neq\ell(\mu)$ then $r_\mu -r _\lambda$ is not Schur positive. Note that by Lemma~\ref{theorem-ribbon-schur} if $S(\lambda) = \{ i_1, i_2, \ldots , i_{\ell(\lambda)-1}\}$ then in any SYT, $T$, of shape $\nu $ and descent set $S(\lambda)$ contributing towards the coefficient of $s_\nu$ in the Schur function expansion of $r_\lambda$, the entries $1,\ldots , i_j$ must appear in the top $j$ rows of $T$. Moreover, $s_\lambda$ appears in the Schur function expansion of $r_\lambda$ with positive coefficient due to the SYT
$$\begin{matrix}
1&2&3&\cdots&i_1\\
i_1+1&\cdots & &i_2\\
\vdots&\ddots&\\
i _{\ell(\lambda)-1}+1&\cdots&N
\end{matrix}.$$Consequently, if $\mu\not\dom \lambda$ then for some $i$ we have
$$\lambda _1 +\cdots + \lambda _i < \mu _1 +\cdots + \mu _i$$and there cannot exist a SYT of shape $\lambda$ and descent set $S(\mu)$ because the numbers $1,\ldots , \mu _1 +\cdots + \mu _i$  must appear in only $\lambda _1 +\cdots + \lambda _i$ boxes. Hence by Lemma~\ref{theorem-ribbon-schur} it follows that $s_\lambda$ is a term in the Schur function expansion of $r_\lambda$ but not $r_\mu$ and so $r_\mu-r_\lambda$ is not Schur positive.

Now note that if $\mu \dom \lambda$ then since
$$\lambda _1 +\cdots + \lambda _i \geq \mu _1 +\cdots + \mu _i$$for all $i$ it follows that $\ell(\lambda )\leq \ell(\mu)$. We need to show that if  $\ell(\lambda )< \ell(\mu)$ then $r_\mu -r_\lambda$ is not Schur positive. Observe by Lemma~\ref{theorem-ribbon-schur} that if $\lambda = (\rho, 1^a)$ where $\rho$ is a partition with $\rho _{\ell(\rho)}>1$ and $a>0$ then $S(\lambda)= \{i_1, i_2, \ldots , i_{\ell(\rho)}, N-a+1, \ldots, N-1\}$ and the lexicographically greatest partition appearing as an index of a Schur function in the Schur function expansion of $r_\lambda$ is $\phi(\lambda)=(N-\ell(\rho)-a+1, \ell(\rho), 1^{a-1})=(N-\ell(\lambda)+1, \ell(\rho), 1^{a-1})$ due to the SYT
$$\begin{matrix}
1&2&3&\cdots&\cdots&i_1&\cdots&i_2&\cdots  &i_{\ell(\rho)}\\
i_1+1&i_2+1&\cdots  &i_{\ell(\rho)-1}+1&N-a+1\\
N-a+2\\
\vdots\\
N
\end{matrix}.$$If $a=0$ then $S(\lambda)= \{i_1, i_2, \ldots , i_{\ell(\rho)-1}\}$ and Lemma~\ref{theorem-ribbon-schur} shows the lexicographically greatest partition appearing as an index of a Schur function in the Schur function expansion of $r_\lambda$ is $\phi(\lambda)=(N-\ell(\rho)+1, \ell(\rho)-1)=(N-\ell(\lambda)+1, \ell(\rho)-1)$ due to the SYT
$$\begin{matrix}
1&2&3&\cdots&\cdots&i_1&\cdots&i_2&\cdots  &N\\
i_1+1&i_2+1&\cdots  &i_{\ell(\rho)-1}+1\\
\end{matrix}.$$Consequently, if $\ell(\lambda)<\ell(\mu)$ we have $\phi(\mu)< _{lex} \phi(\lambda)$ and so $s_{\phi(\lambda)}$ is a term in the Schur function expansion of $r_\lambda$ but not $r_\mu$ and so $r_\mu-r_\lambda$ is not Schur positive.

Now we prove that if   $\mu \dom \lambda$ and $\ell(\lambda)=\ell(\mu)$ then $r_\mu -r _\lambda$ is  Schur positive. Note we only need to show that the result holds for each of the types of cover relation in the dominance order with $\ell(\lambda)=\ell(\mu)=\ell$. For the first cover relation, if $\lambda = (\sigma , a+1, b-1, \tau), \mu = (\sigma , a,b, \tau)$ and $\sigma = (\sigma _1, \ldots , \sigma _s)$,  $\tau = (\tau _1, \ldots , \tau _t)$ with $\sigma _s\geq a+1$, $b-1\geq \tau _1$ and $a\geq b\geq 2$ then $r_\mu -r_\lambda$ is Schur positive by Theorem~\ref{big-posdifference} with 
$r_\mu = \{\sigma , a,b, \tau\, |\,  1^{\ell - 1}\}$ and $r_\lambda = \{\sigma , a+1, b-1, \tau\, | \,  1^{\ell - 1}\}$. For the second cover relation, if  $\lambda = (\sigma , a+1, a^{n-2}, a-1, \tau), \mu = (\sigma , a^n, \tau)$ and $\sigma = (\sigma _1, \ldots , \sigma _s)$,  $\tau = (\tau _1, \ldots , \tau _t)$ with  $\sigma _s\geq a+1$, $a-1\geq \tau _1$ and $n\geq 2$ then $r_\mu -r_\lambda$ is Schur positive by Theorem~\ref{big-posdifference} with 
$r_\mu = \{\sigma , a^n, \tau\, |\,  1^{\ell - 1}\}$ and $r_\lambda = \{\sigma , a+1, a^{n-2}, a-1, \tau\, | \,  1^{\ell - 1}\}$. The proof is now complete.
\end{proof}

 \section{Application to Schubert calculus}
\label{sec:schubert}

In view of the well established link between the algebra of Schur
functions
and Schubert calculus~\cite{Manivel, ECII},
it is appropriate to examine the consequences of Theorem~3.1 in the
Schubert
calculus context.

Let $\Gr(\ell,\C^m)$ denote the Grassmannian of $\ell$-dimensional
subspaces in $\C^m$. The cohomology ring $H^\ast(\Gr(\ell,\C^m),\Z)$
has an additive basis of Schubert classes, $\sigma_\lambda$, indexed by partitions
$\lambda\subseteq (k^\ell)$, where $k=m-\ell$.
The product rule for Schubert classes in $H^\ast(\Gr(\ell,\C^m),\Z)$
takes the following form, see for example~\cite{Manivel}
\begin{equation}
\label{Eq-schubert-prod}
 \sigma_\lambda \cdot \sigma_\mu =\sum_{\nu\subseteq(k^\ell)} c_{\lambda\mu}^\nu\ \sigma_\nu\,.
\end{equation}
This differs from the product rule for Schur functions
\begin{equation}
\label{Eq-schur-prod}
 s_\lambda \cdot s_\mu =\sum_{\nu} c_{\lambda\mu}^\nu\ s_\nu\,
\nonumber
\end{equation}
only in the restriction placed on the partitions $\nu$. The coefficients
in both products are identical, namely the familiar non-negative integer
Littlewood-Richardson coefficients.

One way to evaluate the product of two Schubert classes is to proceed
by Schur function methods in a manner that automatically restricts
partitions $\nu$ to those satisfying the constraint $\nu\subseteq(k^\ell)$.
To this end, consider the skew Schur function $s_{(k^\ell)/\nu}$. 
For any partition $\nu\subseteq(k^\ell)$, we 
let $\nucomp$ denote the partition complementary to $\nu$ in an
$\ell\times k$ rectangle, which is to say
\begin{equation}
\label{Eq-nucomp}
 \nucomp_i=k-\nu_{\ell-i+1}~~~~\text{~~~~for $i=1,2\ldots,\ell$}\,.
\end{equation}
Diagrammatically, the Young diagram of shape $\nucomp$ 
is obtained by rotating that of shape $(k^\ell)/\nu$
through $180^\circ$. 
For example, if $\ell=4$, $k=5$ and $\nu=(5,4,2,1)$,
then $\nucomp=(4,3,1)$ as illustrated by
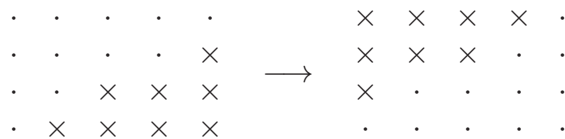
\begin{equation}
\begin{matrix}
 \cdot&\cdot&\cdot&\cdot&\cdot\cr
 \cdot&\cdot&\cdot&\cdot&\times\cr
 \cdot&\cdot&\times&\times&\times\cr
 \cdot&\times&\times&\times&\times\cr
\end{matrix}
\quad\longrightarrow\quad
\begin{matrix}
 \times&\times&\times&\times&\cdot\cr
 \times&\times&\times&\cdot&\cdot\cr
 \times&\cdot&\cdot&\cdot&\cdot\cr
 \cdot&\cdot&\cdot&\cdot&\cdot\cr
\end{matrix}
\nonumber
\end{equation}
where $\times$ indicates a box, and $\cdot$ no box.

With this definition of $\nucomp$, we have
\begin{equation}
\label{Eq-nuskew-nucomp}
  s_{(k^\ell)/\nu}=
  \begin{cases}
    \ s_\nucomp&\text{if $\nu\subseteq(k^\ell)$};\cr
    \ 0&\text{otherwise}.\cr
  \end{cases}
\end{equation}

The first case in \eqref{Eq-nuskew-nucomp} is \cite[Exercise 7.56(a)]{ECII}.
The other case in \eqref{Eq-nuskew-nucomp} is perhaps obvious, but may
be seen formally by noting that the definition of skew Schur 
functions~\cite{Macdonald} implies that
\begin{equation}
   \langle s_{(k^\ell)/\nu}\,,\,s_\rho\rangle=\langle s_{(k^\ell)}\,,\,s_\nu\,s_\rho\rangle
\nonumber
\end{equation} 
for all $\rho$, where the bilinear form $\langle \cdot\,,\,\cdot\rangle$
on symmetric functions is such that $\langle s_\lambda\,,\,s_\mu\rangle=\delta_{\lambda\mu}$.
However, it follows from the Littlewood-Richardson rule
that the right hand side is zero for all $\rho$ if $\nu\not\subseteq(k^\ell)$.
Hence, in such a case $s_{(k^\ell)/\nu}=0$, as claimed in \eqref{Eq-nuskew-nucomp}.

Now consider the skew Schur function $s_{\mucomp/\lambda}$. This
can be expanded in terms of Schur functions in two ways as follows:
\begin{equation}
\label{Eq-skew-mu-lambda1}
s_{\mucomp/\lambda}=\sum_\rho\ c_{\lambda\rho}^\mucomp\ s_\rho\,
\nonumber
\end{equation}
and
\begin{equation}
\label{Eq-skew-mu-lambda2}
\begin{array}{lclcl}
s_{\mucomp/\lambda}&=&\sum_\kappa\ \langle s_{\mucomp/\lambda}\,,\,s_\kappa\rangle\ s_\kappa
&=&\sum_\kappa\ \langle s_\mucomp\,,\,s_\kappa\,s_\lambda\rangle\ s_\kappa\cr\cr
&=&\sum_\kappa\ \langle s_{(k^\ell)/\mu}\,,\,s_\kappa\,s_\lambda\rangle\ s_\kappa
&=&\sum_\kappa\ \langle s_{(k^\ell)}\,,\,s_\kappa\,s_\lambda\,s_\mu\rangle\ s_\kappa\cr\cr
&=&\sum_\kappa  \sum_\nu\ c_{\lambda\mu}^\nu\ \langle s_{(k^\ell)}\,,s_\kappa\,s_\nu\rangle\ s_\kappa
&=&\sum_\kappa  \sum_\nu\ c_{\lambda\mu}^\nu\ \langle s_{(k^\ell)/\nu}\,,s_\kappa\rangle\ s_\kappa\cr\cr
&=&\sum_\kappa  \sum_{\nu\subseteq(k^\ell)}\ c_{\lambda\mu}^\nu\ \langle s_{\nucomp}\,,s_\kappa\rangle\ s_\kappa
&=&\sum_{\nu\subseteq(k^\ell)}\ c_{\lambda\mu}^\nu\ s_{\nucomp}\,.\cr
\end{array}
\nonumber
\end{equation}
On comparing these two expressions, it can be seen that 
to evaluate the product $\sigma_\lambda\cdot\sigma_\mu$ of Schubert classes
given in \eqref{Eq-schubert-prod} one merely has to expand  
the skew Schur function $s_{\mucomp/\lambda}$ in terms of Schur functions.
To be more precise 
\begin{equation}
\label{Eq-schubert-prod-new}
   \sigma_\lambda \cdot \sigma_\mu 
   =\sum_\nu\ c_{\lambda,\nucomp}^{\mucomp}\ \sigma_\nu\,
\end{equation}
as proved also in~\cite{Gutschwager} and implicitly in \cite{ThomasYong}.

Moreover, the skew shape corresponding to
$\mucomp/\lambda$ is that of the diagram obtained
by placing the Young diagram of shape $\lambda$ in the top lefthand
corner of an $\ell\times k$ rectangle, and that of $\mu$ rotated through $180^\circ$
in the bottom righthand corner and deleting both sets of boxes from
the Young diagram of $(k^\ell)$. This is illustrated in the case $\ell=5$, $k=6$,
$\lambda=(5,2,1)$ and $\mu=(5,4,2)$ as shown below, where each $\lambda$
signifies a box of the Young diagram of shape $\lambda$, each $\mu$ a box of the
rotated Young diagram of shape $\mu$. 
\begin{equation}
\begin{array}{cccccc}
\lambda&\lambda&\lambda&\lambda&\lambda&\alpha\\
\lambda&\lambda&\alpha &\alpha &\alpha &\alpha\\
\lambda&\alpha &\alpha &\alpha &\mu    &\mu\\
\alpha &\alpha &\mu    &\mu    &\mu    &\mu\\
\alpha &\mu    &\mu    &\mu    &\mu    &\mu\\
\end{array}
\nonumber
\end{equation}
In a case, such as this, where the Young diagrams of shape
$\lambda$ and $\mu$ do not overlap, so that $s_{\mucomp/\lambda}$
is non-zero, the remaining boxes labelled $\alpha$ are those of the required skew
shape $\mucomp/\lambda$. In this example, we have 
$s_{\mucomp/\lambda}=s_{(6,6,4,2,1)/(5,2,1)}$, which in the
overlap notation of \eqref{Eq-sfn-overlap} is  $\{1,4,3,2,1\,|\,1,2,1,1\}$.

As a direct consequence of this, we are in a position to apply our previous results
to the question of what we shall call the Schubert positivity of 
differences of products of Schubert classes. In this connection
we make the following definition.

\begin{definition}
If an element $\sigma \in H^\ast(\Gr(\ell,\C^m),\Z)$ can be written as a non-negative
linear combination of Schubert classes then we say  that $\sigma$ is \emph{Schubert positive}.
\end{definition}

Then we have the following theorem on differences of products of Schubert classes.

\begin{theorem}
\label{The-schubert} 
Let $\sigma,\tau,\ov\sigma,\ov\tau$ satisfy Hypothesis~\ref{ovsigma}, and let
the partitions $\lambda,\mu,\kappa,\rho\subseteq(k^\ell)$ be such that
in the overlap notation of \eqref{Eq-skew-overlap}
either
\begin{equation}
\begin{array}{rcl} 
\mucomp/\lambda&=&(\sigma,a,b,\tau\,|\,\ov\sigma,1,\ov\tau)\cr
\rhocomp/\kappa&=&(\sigma,a+1,b-1,\tau\,|\,\ov\sigma,1,\ov\tau)\cr
\end{array}
\nonumber
\end{equation}
with $a\geq b\geq 2$, or
\begin{equation}
\begin{array}{rcl} 
\mucomp/\lambda&=&(\sigma,a^n,\tau\,|\,\ov\sigma,1^{n-1},\ov\tau)\cr
\rhocomp/\kappa&=&(\sigma,a+1,a^{n-2},a-1,\tau\,|\,\ov\sigma,1^{n-1},\ov\tau)\cr
\end{array}
\nonumber
\end{equation}
with $a\geq2$, then the product of Schubert classes in $H^\ast(\Gr(\ell,\C^{k+l}),\Z)$
is such that
\begin{equation}
                  \sigma_\lambda\cdot\sigma_\mu - \sigma_\kappa\cdot\sigma_\rho
                  \nonumber
\end{equation}
is Schubert positive.
\end{theorem}

\begin{proof} Theorem~\ref{big-posdifference} implies that in each case 
$s_{\mucomp/\lambda}-s_{\rhocomp/\kappa}$
is Schur positive, which is to say
$c_{\lambda,\nucomp}^{\mucomp}-c_{\kappa,\nucomp}^{\rhocomp}$
is non-negative for each $\nu$. The required result then follows
from \eqref{Eq-schubert-prod-new}.
\end{proof}

As a further application, it is instructive to consider 
the implications of our Theorem~\ref{positive-ribbon-theorem} 
on differences of ribbon Schur functions. 

\begin{theorem}
\label{theorem-schubert-posdifference}
Let the partitions $\lambda,\mu,\kappa,\rho\subseteq(k^\ell)$ be such that
$\mucomp/\lambda=(\alpha\,|\,1^{\ell-1})$ and 
$\rhocomp/\kappa=(\beta\,|\,1^{\ell-1})$ with
$\alpha$ and $\beta$ partitions of lengths $\ell(\alpha)=\ell(\beta)=\ell$ and 
 $|\alpha|=|\beta|=N\leq\ell+k-1$, such that $\alpha\dom\beta$. Then
the difference of products of Schubert classes 
\begin{equation}
   \sigma_\lambda\cdot\sigma_\mu - \sigma_\kappa\cdot\sigma_\rho
   \nonumber
\end{equation}
is Schubert positive in $H^\ast(\Gr(\ell,\C^{k+l}),\Z)$.
\end{theorem}

\begin{proof}
Since $\mucomp/\lambda=(\alpha\,|\,1^{\ell-1})$ is a ribbon, it follows from 
Lemma~\ref{theorem-ribbon-schur} that
\begin{equation}
r_\alpha=s_{\mucomp/\lambda} 
=\sum_{\nu\subseteq(k^\ell)} c_{\lambda,\nucomp}^{\mucomp}\ s_{\nucomp}
=\sum_{\nu\subseteq(k^\ell)} d_{\nucomp\alpha}\ s_{\nucomp}\,.
\nonumber
\end{equation}
Using this in \eqref{Eq-schubert-prod-new} gives
\begin{equation}
\label{Eq-schubert-ribbon}
 \sigma_\lambda\cdot\sigma_\mu =\sum_{\nu\subseteq(k^l)} d_{\nucomp\alpha}\ \sigma_\nu\,.
\end{equation} 
A similar result applies to $\sigma_\kappa\cdot\sigma_\rho$, so that
from \eqref{Eq-schubert-prod-new} we obtain
\begin{equation}
\sigma_\lambda\cdot\sigma_\mu-\sigma_\kappa\cdot\sigma_\rho 
=\sum_{\nu\subseteq(k^\ell)}(d_{\nucomp\alpha}-d_{\nucomp\beta})\ \sigma_\nu \,.
\label{Eq-schubert-ribbon-diff}
\nonumber
\end{equation}
By hypothesis, $|\alpha|=|\beta|=N$, $\ell(\alpha)=\ell(\beta)=\ell$
and $\alpha\dom\beta$, so that Theorem~\ref{positive-ribbon-theorem} implies that 
\begin{equation}
\label{Eq-schur-ribbon-diff}
r_\alpha-r_\beta=\sum_\eta (d_{\eta\alpha}-d_{\eta\beta})\ s_\eta 
\nonumber
\end{equation}
is Schur positive, and thus $d_{\eta\alpha}-d_{\eta\beta}\geq0$ for all
$\eta$. This ensures, in turn, that 
$\sigma_\lambda\cdot\sigma_\mu-\sigma_\kappa\cdot\sigma_\rho$ is 
Schubert positive in $H^\ast(\Gr(\ell,\C^{k+\ell}),\Z)$.
\end{proof}

As a special case, consider $\alpha=(2^{n+2})$ and $\beta=(3,2^n,1)$.
Setting $\ell=n+2$ and $k=n+3$, so that $N=2n+4=\ell+k-1$ as required, we have 
$\lambda=\kappa=(n+1,n,\ldots,1)=:\delta_{n+1}$, the \emph{staircase partition of length} $n+1$,  
while $\mu=(n,n-1,\ldots,1)=\delta_n$ and $\nu=(n+2,n+1,\ldots,2)=\delta_{n+2}\backslash1$.
With this notation, the above Theorem~\ref{theorem-schubert-posdifference} implies that 
\begin{equation}
\label{Eq-delta}
         \sigma_{\delta_{n+1}}\cdot \sigma_{\delta_{n+1}} 
         - \sigma_{\delta_{n}}\cdot \sigma_{\delta_{n+2}\backslash1}
\end{equation}
is Schubert positive in $H^\ast(\Gr(n+2,\C^{2n+5}),\Z)$. For example, the $n=2$ 
case of this implies that
\begin{equation}
          \sigma_{(3,2,1)}\cdot\sigma_{(3,2,1)}-\sigma_{(2,1)}\cdot\sigma_{(4,3,2)}
\nonumber     
\end{equation}
is Schubert positive in $H^\ast(\Gr(4,\C^{9}),\Z)$,

Although the Schubert positivity of \eqref{Eq-delta} can also be established
by noting that \cite{LamPostnikovPylyavskyy}
\begin{equation}
         s_{\delta_{n+1}}\cdot s_{\delta_{n+1}} 
         - s_{\delta_{n}}\cdot s_{\delta_{n+2}\backslash1}
\end{equation}
is itself Schur positive, this property of Schur functions is 
not easy to derive.

Moreover, in contrast to this, even though
\begin{equation}
  s_{(4,3,2)}\cdot s_{(3,2,1)}-s_{(5,4,2)}\cdot s_{(3,1)}
  =- s_{(8,5,2)}\ +\ \cdots\ +\ s_{(4,3,3,2,2,1)}
\end{equation}
is not Schur  positive, the corresponding
product of Schubert classes
\begin{equation}
  \sigma_{(4,3,2)}\cdot\sigma_{(3,2,1)}-\sigma_{(5,4,2)}\cdot\sigma_{(3,1)}
\nonumber
\end{equation}
is Schubert positive in $H^\ast(\Gr(4,\C^{10}),\Z)$.
This
is a consequence of the fact that for $\ell=4$ and $k=6$ we know
\begin{equation}
   s_{(4,3,2)^c/(3,2,1)}-s_{(5,4,2)^c/(3,1)}=r_{(3,2,2,2)}-r_{(3,3,2,1)}
\end{equation}
is Schur positive, by Example~\ref{useful-ex}.

\section{Acknowledgements} The authors would like to thank Peter McNamara for 
sharing his conjecture with us and for helpful discussions, and  the referees for their constructive comments.

\end{document}